\documentclass[a4paper,11pt,twoside]{article}
\usepackage[T1]{fontenc}
\usepackage[utf8]{inputenc}
\usepackage[english]{babel}
\usepackage{amsmath, amssymb, amsthm, amsfonts, mathtools, mathrsfs}	
\usepackage{enumerate}
\usepackage[shortlabels]{enumitem}
\usepackage{booktabs, caption, graphicx, subcaption, wrapfig} 	
\usepackage{geometry}
\usepackage{xcolor}

\usepackage{cite} 
\usepackage{upgreek}
\usepackage{siunitx}
\usepackage{bbm}

\usepackage{fancyhdr} 

\newcommand{\R}{\mathbb{R}}
\newcommand{\N}{\mathbb{N}}

\newcommand{\C}{\mathbb{C}}

\numberwithin{equation}{section}

\theoremstyle{plain}
\newtheorem{theorem}{Theorem}[section] 
\newtheorem{proposition}[theorem]{Proposition} 
\newtheorem{lemma}[theorem]{Lemma}

\theoremstyle{definition}

\newtheorem{remark}[theorem]{Remark}

\DeclarePairedDelimiter{\abs}{\lvert}{\rvert}
\DeclarePairedDelimiter{\norm}{\lVert}{\rVert}

\DeclareMathOperator*{\argmin}{arg\,min}

\newcommand{\eps}{\varepsilon}
\renewcommand{\phi}{\varphi}
\renewcommand{\bar}{\overline}
\renewcommand{\vec}{\boldsymbol}

\def\genspazio #1#2#3#4#5{#1^{#2}(#5,#4;#3)}
\def\spazio #1#2#3{\genspazio {#1}{#2}{#3}T0}

\def\LT {\spazio L}
\def\HT {\spazio H}
\def\C #1#2{\mathcal{C}^{#1}([0,T];#2)}

\def\Lx #1{L^{#1}(\Omega)}

\def\Lqt #1{L^{#1}(Q_T)}
\def\Hx #1{H^{#1}(\Omega)}

\def\Huo {H^1_0(\Omega)}

\def\matteo #1{{\color{blue}#1}}

\def\ls{<}
\def\gs{>}
\def\mezzo {\frac{1}{2}}

\def\de {\mathrm{d}}
\def\D {\mathrm{D}}

\def\n {\vec{n}}
\def\phib {\bar{\phi}}
\def\pb {\bar{p}}
\def\sigmab {\bar{\sigma}}
\def\hh {\mathbbm{h}}

\def\Iad {\mathcal{I}_{\text{ad}}}

\def\HH {\mathbb{H}}
\def\VV {\mathbb{V}}
\def\vpsi {\vec{\uppsi}}

\def\vh {\vec{h}}

\def\Cb {\bar{C}}

\def\Rcal {\mathcal{R}}
\def\Jcal {\mathcal{J}}

\def\triplein {(\phi_0, \sigma_0, p_0)}
\def\tripleinin {(\phi^0_0, \sigma^0_0, p^0_0)}
\def\tripleinbar {(\bar{\phi}_0, \bar{\sigma}_0, \bar{p}_0)}

\def\triplefin {(\phi(T), \sigma(T), p(T))}

\def\Ccal {\mathcal{C}}

\def\phimeas {\phi_{\text{meas}}}
\def\sigmameas {\sigma_{\text{meas}}}
\def\pmeas  {p_{\text{meas}}}
\def\triplemeas {(\phimeas, \sigmameas, \pmeas)}
\def\vpsimeas {\vec{\uppsi}_{\text{meas}}}
\def\phitil {\tilde{\phi}}
\def\sigmatil {\tilde{\sigma}}
\def\ptil {\tilde{p}}
\def\tripletil {(\phitil, \sigmatil, \ptil)}

\usepackage{hyperref}

\begin{document}
	
	\begin{center}
		
		\LARGE{\textbf{Iterative algorithms for the reconstruction of \\ early states of prostate cancer growth}}
		
		\vskip0.5cm
		
		\large{\textsc{Elena Beretta$^1$}} \\
		\normalsize{e-mail: \texttt{eb147@nyu.edu}} \\
		\vskip0.35cm
		
		\large{\textsc{Cecilia Cavaterra$^2$}} \\
		\normalsize{e-mail: \texttt{cecilia.cavaterra@unimi.it}} \\
		\vskip0.35cm
		
		\large{\textsc{Matteo Fornoni$^3$}} \\
		\normalsize{e-mail: \texttt{matteo.fornoni01@universitadipavia.it}} \\
		\vskip0.35cm

        \large{\textsc{Guillermo Lorenzo$^4$}} \\
		\normalsize{e-mail: \texttt{guillermo.lorenzo.gomez@sergas.es}} \\
		\vskip0.35cm
		
		\large{\textsc{Elisabetta Rocca$^5$}} \\
		\normalsize{e-mail: \texttt{elisabetta.rocca@unipv.it}} \\
		\vskip0.35cm
		
		\footnotesize{$^1$Division of Science, New York University Abu Dhabi, Saadiyat Island, Abu Dhabi, United Arab Emirates}
		\vskip0.1cm
		
		\footnotesize{$^2$Department of Mathematics ‘‘F. Enriques’’, University of Milan, 20133 Milan, Italy \\ \& IMATI-C.N.R., 27100 Pavia, Italy}
		\vskip0.1cm
		
		\footnotesize{$^3$Department of Mathematics ‘‘F. Casorati’’, University of Pavia, 27100 Pavia, Italy}
		\vskip0.1cm

        \footnotesize{$^4$Health Research Institute of Santiago de Compostela, 15706 Santiago de Compostela, Spain \& Oden Institute for Computational Engineering and Sciences, The University of Texas at Austin, Austin, Texas 78712 USA}
		\vskip0.1cm

        \footnotesize{$^5$Department of Mathematics ‘‘F. Casorati’’, University of Pavia, 27100 Pavia, Italy \\ \& IMATI-C.N.R., 27100 Pavia, Italy}
		\vskip0.35cm
		
	\end{center}

    \vskip2mm
	\begin{abstract}
 
     \noindent
     The development of mathematical models of cancer informed by time-resolved measurements has enabled personalised predictions of tumour growth and treatment response. However, frequent cancer monitoring is rare, and many tumours are treated soon after diagnosis with limited data. To improve the predictive capabilities of cancer models, we investigate the problem of recovering earlier tumour states from a single spatial measurement at a later time. Focusing on prostate cancer, we describe tumour dynamics using a phase-field model coupled with two reaction-diffusion equations for a nutrient and the local prostate-specific antigen. We generate synthetic data using a discretisation based on Isogeometric Analysis. Then, building on our previous analytical work, \cite{BCFLR2024}, we propose an iterative reconstruction algorithm based on the Landweber scheme, showing local convergence with quantitative rates and exploring an adaptive step size that leads to faster reconstruction algorithms. Finally, we run simulations demonstrating high-quality reconstructions even with long time horizons and noisy data.

		\vskip2mm
		
		
		\vskip2mm
		
	\end{abstract}


\pagestyle{fancy}
\fancyhf{}	
\fancyhead[EL]{\thepage}
\fancyhead[ER]{\textsc{Beretta -- Cavaterra -- Fornoni -- Lorenzo -- Rocca}} 
\fancyhead[OL]{\textsc{Reconstructing early states of prostate cancer}} 
\fancyhead[OR]{\thepage}

\renewcommand{\headrulewidth}{0pt}
\setlength{\headheight}{5mm}

\thispagestyle{empty} 

\section{Introduction}
Mathematical modelling of cancer dynamics during growth and treatment contributes to gaining insight into the biophysical mechanisms underlying these phenomena \cite{yin2019review,Lorenzo2022_review,kazerouni2020integrating}.
Some of these models have also been shown to enable the calculation of personalised tumour forecasts that can assist physicians in clinical decision-making \cite{Lorenzo2022_review,Wu2022,Lorenzo2024,Hormuth2021,Wong2016}.
Towards this end, these models employ the clinical and imaging data that are regularly collected from patients according to cancer management protocols in order to inform diagnosis, prognosis, patient triaging, and treatment selection \cite{Mottet2021,kazerouni2020integrating,Lorenzo2022_review,Lorenzo2024,Hormuth2021,Wu2022}.
In particular, personalised tumour forecasting requires data at several time points to initialise the model, calibrate its parameters, and then produce the patient-specific prediction \cite{Lorenzo2022_review,Wu2022,Lorenzo2024,Hormuth2021,Wong2016}.
In some clinical scenarios, these longitudinal data are available because they are currently used to assess tumour changes that correlate with clinical endpoints of interest  (e.g., progression to higher-risk disease, treatment failure, survival) \cite{Giganti2021,tudorica2016early}.
For example, longitudinal tumour measurements are leveraged to identify progression to more malignant stages of newly diagnosed untreated prostate cancer during active surveillance \cite{Lorenzo2024}, as well as to assess therapeutic response during neoadjuvant chemotherapy of breast cancer \cite{Wu2022} and chemoradiation of high-grade glioma \cite{Hormuth2021}.

Despite the increased information on tumour dynamics provided by longitudinal cancer monitoring,  follow-up strategies may not include sufficiently frequent tests to accurately assess changes in tumour status and inform mathematical models \cite{Lorenzo2022_review,kazerouni2020integrating,chaudhuri2023predictive,Lorenzo2024}.
Additionally, some tumours are only diagnosed once they have developed sufficiently to produce symptoms or enable detection with standard-of-care screening methods \cite{Lorenzo2022_review,kazerouni2020integrating,chaudhuri2023predictive,Lorenzo2024}.
In these situations, the estimation of recent patient-specific tumour dynamics from a single dataset could provide physicians with more accurate estimations of prognosis to guide clinical decision-making.
This highly-coveted computational capability could further contribute to reduce treatment excesses and deficiencies, which can respectively affect the patients’ quality of life and life expectancy \cite{gupta2022systemic,ziu2020role,neal2020ten}.
Despite the inherent difficulty in estimating previous tumour dynamics from a single dataset, model-constrained reconstruction algorithms are a viable methodology to address this computational challenge and crucial clinical demand \cite{BCFLR2024,SSMB2020}. 
For example, following the collection of a medical imaging measurement to characterise the tumour morphology (e.g., at diagnosis, after treatment), a reconstruction algorithm based on an adequate mathematical model can identify the tumour status at an earlier time  \cite{Lorenzo2022_review,SSMB2020,chaudhuri2023predictive,BCFLR2024}.
The resulting reconstruction of tumour dynamics can then be used to better estimate biomarkers of progression or therapeutic response, as well as to locate the region of the host organ where the tumour originated.
Indeed, the latter information can have important relevance for some tumours, for instance, in the brain and the prostate \cite{ali2024tale, jungk2019location}.
Furthermore, the reconstruction of early tumour states is also useful in some preclinical scenarios, for example, before the onset of experimental procedures or in sparse data collection regimens with large changes in tumour dynamics in between consecutive measurements \cite{kazerouni2020integrating,Lima2022,Yang2022}.

Here, we investigate model-constrained reconstruction algorithms informed by a single spatial measurement (e.g., \emph{via} medical imaging), and we apply them to the estimation of early stages of newly diagnosed prostate cancer.
The detection of these tumours relies on a multiparametric magnetic resonance imaging scan, which is motivated by increasing values of the serum prostate-specific antigen (PSA; the main blood biomarker of prostate cancer) and guides the ensuing biopsy to confirm the disease histopathologically \cite{Mottet2021,Lorenzo2024}.
The clinical management options for newly diagnosed prostate cancer include active surveillance for lower-risk disease (i.e., monitoring of indolent cases with longitudinal MRI, PSA, and biopsies until therapeutic intervention is needed) and treatment for higher-risk tumours (e.g., surgery, radiotherapy) \cite{Mottet2021}.
Therefore, the accurate identification of the clinical risk of prostate cancer at diagnosis is fundamental to guide the management of the disease, and reconstruction of earlier disease dynamics prior to diagnosis can help in this crucial triaging step \cite{Lorenzo2024,Giganti2021,BCFLR2024}.
Towards this end, our reconstruction algorithms employ a mathematical model of prostate cancer growth that relies on the phase-field method, and that has been studied analytically and computationally in our previous works \cite{CGLMRR2019,CGLMRR2021,BCFLR2024}.
Phase-field models constitute an established spatiotemporal continuous formulation of the dynamics of the tumour geometry, which has also been used in the context of optimal control problems \cite{Cavaterra2019, Colli2017, EK2020, F2024, FLS2021, Garcke2018, Lorenzo2016}.
Thus, our model relies on a continuous phase field $\phi$, such that $\phi\approx0$ in healthy prostatic tissue and $\phi\approx1$ in the tumour.
The tumour-healthy tissue interface exhibits a smooth and steep profile, in which the tumour phase field rapidly varies between 0 and 1.
Our model further assumes that tumour growth is driven by a generic nutrient $\sigma$ (e.g.,~oxygen or glucose), whose concentration follows reaction-diffusion dynamics.
Moreover, another reaction-diffusion equation describes the local dynamics of tissue PSA $p$, which represents the PSA leaked to the bloodstream per unit volume of prostatic tissue \cite{Lorenzo2016}.


More specifically, the problem we want to address is a backward inverse problem for a nonlinear system of parabolic equations. 
Following the notation used in \cite{BCFLR2024}, we call $\mathcal{R}: (\phi_0, \sigma_0, p_0) \mapsto (\phi(T), \sigma(T), p(T))$ the nonlinear solution operator which associates to any initial data $(\phi_0, \sigma_0, p_0)$ the value of the solution at the final time $(\phi(T), \sigma(T), p(T))$. 
Then, given some measurements $(\phimeas, \sigmameas, \pmeas)$, we can formulate our inverse problem as that of finding initial data $(\phi_0, \sigma_0, p_0)$ such that the solution $(\phi, \sigma, p)$ to our model satisfies $\phi(T) = \phimeas$, $\sigma(T) = \sigmameas$ and $p(T) = \pmeas$. 
For our theoretical discussion, we consider measurements of all three variables $\phimeas$, $\sigmameas$ and $\pmeas$ at the terminal time. 
However, in our simulations, we focus only on $\phimeas$ as a specific case where the datum comes from the MRI at diagnosis. 
This would be the most common use of this procedure in a practical scenario.
We mention that, even in the linear case, backward inverse problems of this kind, are well-known to be severely ill-posed, with very weak conditional logarithmic stability from the data, see for example \cite{Hao2011, Payne75, LRS, isakov}. 
As such, while being interesting problems to study, due to the underlying ill-posedness, in the context of tumour growth models the literature on backward inverse problems of this kind is quite lacking.
We mention for instance \cite{JBJA2019}, where the authors also use a similar Landweber-type method to recover the initial datum in a Fisher--Kolmogorov model for brain tumour growth. 
Albeit getting some good reconstruction results in $3D$, the authors do not address the analytical questions related to the inverse problem, as was done in \cite{BCFLR2024}, even if the model is simpler. 
Another approach is followed in \cite{SSMB2020, subramanian2022ensemble}, where the authors still consider a Fisher--Kolmogorov model (with some additional couplings) and regularise the unknown initial data by assuming it is representable by a sum of Gaussians centred in some points belonging to a lattice. 
They further use a Tikhonov-like regularisation approach, by additionally enforcing sparsity, to reconstruct not only the initial datum but also some model coefficients. 
With our contributions, we aim to show that a good reconstruction can be guaranteed both theoretically and numerically without assuming restrictive hypotheses on the initial data, even for more complex models.

As a preliminary step, it is of paramount importance to regularise the problem, by establishing physically relevant a priori assumptions on the unknown initial data that lead to better, possibly  Lipschitz, dependence of the initial data on the measurements. 
Following similar ideas, in \cite{BCFLR2024} the authors applied this procedure to the above-introduced inverse problem, proving a quantitative Lipschitz stability estimate for the reconstruction of the initial data.
Starting from these results, we now propose an iterative algorithm to solve the identification problem numerically. 
Such reconstruction algorithm is based on the Landweber iteration scheme, a widely used method in approximating and regularising solutions to inverse problems \cite{DQS2012, HNS1995, kaltenbacher:neubauer:scherzer}.
It is well-known that, when good stability estimates are available, this kind of iterative scheme is locally convergent and acts as a regularisation method in the presence of noise on the measurements \cite{kaltenbacher:neubauer:scherzer}.
Indeed, in our case, we are able to prove local convergence of the Landweber algorithm (Theorem \ref{thm:landweber}) both with a constant step size and an adaptive one, based on the steepest descent.
However, since the problem is ill-posed and all stability constants depend exponentially on the final time, such step size choices do not work well enough if $T$ is too large. 
In such scenarios, we employ a different step size recently introduced in \cite{MM2019adaptive}, which, despite lacking a comprehensive theoretical guarantee, has proven to be surprisingly accurate in the reconstruction.
Hence, by using the Landweber scheme with this new choice of the step size, we are even able to get good reconstruction results when the terminal time $T$ reaches one year within a low number of iterations. 
Under this setup, we conduct various numerical experiments to validate the analytical results and explore the behaviour of the reconstruction algorithm.
For the numerical discretisation, we employ Isogeometric Analysis (IGA) in space \cite{Cottrell2009} and the generalised $\alpha$-method in time \cite{Chung1993}.
We then implement the above-mentioned Landweber scheme with different step-size choices, depending on the chosen time horizon.
All our experiments are conducted considering a tumour growing in a square tissue patch and using synthetic ground truth data generated by the forward phase-field model.
In particular, we focus only on the reconstruction of the tumour variable, even though reconstructions of the other variables are also possible within our framework.
Through our simulation study, we show that, if $T$ is relatively small (i.e., less than a month), the Landweber scheme with the steepest descent adaptive step size produces high-quality reconstructions and matches the theoretical results on the order of convergence. 
Noticing that the number of iterations quickly rises if $T$ grows larger, we switched to the adaptive step size introduced in \cite{MM2019adaptive} to conduct experiments for longer time horizons (e.g., several months to a year). 
Thus, by applying some additional care in the selection of the initial guess, we also were able to get good reconstructions even in this much more challenging case. 
Moreover, we additionally assessed the performance of our algorithms in the presence of Gaussian noise on the terminal data, getting faithful reconstructions also in this case.

The paper is organised as follows. 
In Section~\ref{sec:mathmodel} we thoroughly introduce our mathematical model, by referring to \cite{CGLMRR2019}. 
Section~\ref{sec:analysis} is devoted to the recollection of the theoretical results obtained in \cite{BCFLR2024} on the analysis of the inverse problem. 
Section~\ref{sec:landweber} addresses the problem of the approximation of the solution through the use of a Landweber iteration scheme, by also proving some rigorous results on its convergence properties. 
In Section~\ref{sec:numeth} we present our computational methods and, then, in Section~\ref{sec:sims} we show and comment on some representative simulations.
Finally, in Section~\ref{sec:discussion} we discuss our results, possible limitations, and future research directions.


\section{Mathematical model}
\label{sec:mathmodel}

We leverage the phase-field model of prostate cancer growth that was presented and analysed in Refs.~\cite{CGLMRR2019, CGLMRR2021, BCFLR2024}.
The interested reader is referred to these previous works for a detailed presentation of the biological phenomena included in the model. 
The only difference with respect to the original formulation in \cite{CGLMRR2019} is that we have eliminated the terms describing cytotoxic and antiangiogenic treatment effects.
The rationale for this modification is that here we are interested in reconstructing early stages of prostate cancer growth before diagnosis and, thus, before the onset of treatments.
In the following, we briefly present the main equations of the model and outline their components. 
Let $\Omega \subset \R^N$, $N=2,3$, be an open and bounded domain with $\mathcal{C}^2$ boundary and outward unit normal vector $\n$. Let $T \gs 0$ be a time horizon, and denote $Q_t := \Omega \times (0,t)$ and $\Sigma_t = \partial \Omega \times (0,t)$ for any $t \in (0,T]$. 
Then, the prostate cancer model can be formulated as
\begin{alignat}{2}
    & \partial_t \phi = \lambda \Delta \phi - F'(\phi) + m(\sigma) \hh'(\phi) 
    && \hbox{in $Q_T$,} \label{eq:phi} \\
    & \partial_t \sigma = \eta \Delta \sigma 
    + S_h (1-\phi) + S_c \phi 
    - \left( \gamma_h (1-\phi) + \gamma_c \phi  \right) \sigma
    \qquad && \hbox{in $Q_T$,} \label{eq:sigma} \\
    & \partial_t p = D \Delta p  
    + \alpha_h (1 - \phi)+ \alpha_c \phi - \gamma_p p
    && \hbox{in $Q_T$,} \label{eq:p} \\
    & \phi = 0, \quad \partial_{\n} \sigma = 0, \quad \partial_{\n} p = 0 
    && \hbox{on $\Sigma_T$,} \label{bc} \\
    & \phi(0,x) = \phi_0, \quad \sigma(0,x) = \sigma_0, \quad p(0,x) = p_0
    && \hbox{in $\Omega$.} \label{ic}
\end{alignat}
The differential operators in these equations are defined as follows: the subscript $t$ denotes partial differentiation with respect to time, $\Delta$ is the Laplace operator with respect to the space variables, $F'$ denotes the derivative of $F$, and the subscript $\vec{n}$ indicates the outward normal derivative to the domain boundary $\partial\Omega$.

In Eqs.~\eqref{eq:phi}-\eqref{ic}, $\phi$ is a phase field that identifies the spatial regions occupied by healthy tissue ($\phi\approx 0$) and the tumour ($\phi\approx 1$).
In the context of cancer phase-field models, $\lambda=M\ell^2$ is the tumour cell diffusion coefficient, $F(\phi) = M \phi^2 (1-\phi)^2$ is a regular double-well potential, $\hh(\phi) = M \phi^2 (3 -2 \phi)$ is an interpolation function, $M$ is a positive constant associated with tumour cell mobility, and $\ell$ is a positive constant denoting the interface length scale \cite{Xu2016,gomez2018computational}.
Additionally, the spatiotemporal dynamics of the tumour phase field in Eq.~\eqref{eq:phi} is driven by a nutrient field $\sigma$ \emph{via} the tilting function $m(\sigma)$, which represents net tumour cell proliferation.
This function is defined as
\begin{equation}\label{msigma}
m(\sigma )=m_{ref}\left( \frac{\rho +A}{2}+\frac{\rho -A}{\pi }\arctan \left( \frac{\sigma -\sigma _{l}}{\sigma _{r}}\right) \right),
\end{equation}
\noindent where $m_{ref}$, $\rho$, and $A$ are constants that respectively represent a positive scaling factor and two non-dimensional indices associated with tumour cell proliferation and death.
We further define $\rho=\frac{K_\rho}{\bar{K_\rho}}$ and $A=-\frac{K_A}{\bar{K_A}}$, where $K_\rho$ and $K_A$ are the tumour cell proliferation and death rates while $\bar{K_\rho}$ and $\bar{K_A}$ are their corresponding scaling reference values.
The constants $K_\rho, \bar{K_\rho},K_A,$ and $\bar{K_A}$ are all positive.
Moreover, in Eq.~\eqref{msigma}, the positive constants $\sigma_l$ and $\sigma_r$ denote a reference and a threshold value for the nutrient concentration in describing the dependence of the proliferation activity of the tumour on the nutrient availability.

The nutrient follows reaction-diffusion dynamics according to Eq.~\eqref{eq:sigma}, where $\eta$ denotes the nutrient diffusivity.
The first two reaction terms in the right-hand side of Eq.~\eqref{eq:sigma} represent the nutrient supply to healthy and cancerous tissue governed by the rates $S_h$ and $S_c$.
Similarly, the last two reaction terms in this equation model nutrient consumption in healthy and cancerous tissue mediated by the rates $\gamma_h$ and $\gamma_c$.
In Eq.~\eqref{eq:p}, we also choose reaction-diffusion dynamics for the tissue PSA $p$ \cite{Lorenzo2016}.
The tissue PSA diffusivity is denoted by $D$.
The first two reaction terms in the right-hand side of Eq.~\eqref{eq:p} describe the production of tissue PSA in healthy and cancerous tissue at rates $\alpha_h$ and $\alpha_c$, respectively.
The last reaction term in this equation represents the natural decay of tissue PSA with rate $\gamma_p$.
Hence, to recover the serum PSA value $P_s$ used in clinical practice, it suffices to integrate the tissue PSA over the spatial domain (i.e., $P_s=\int_\Omega p \, \de x$).
Of note, all the parameters in Eqs.~\eqref{eq:sigma} and \eqref{eq:p} are positive and constant in this work.
Furthermore, for simplicity of exposition, we sometimes adopt the compact notation 
\[
    \gamma_{ch} := \gamma_c - \gamma_h, \quad S_{ch} := S_c - S_h, \quad \alpha_{ch} := \alpha_c - \alpha_h.
\]

Finally, Eqs.~\eqref{bc} define the boundary conditions and Eqs.~\eqref{ic} provide the initial conditions of the model.
In particular, we choose no-flux boundary conditions for the nutrient and the tissue PSA, while we define zero-valued Dirichlet boundary conditions for the tumour phase field.
Hence, we assume that the tumour is confined within the domain, which aligns with the organ-confined stage of the majority of newly-diagnosed prostate cancer cases \cite{Mottet2021}.
Additionally, the functions $\phi_0$, $\sigma_0$, and $p_0$ are spatial maps at $t=0$.

\section{Analytical results}
\label{sec:analysis}

We first define the following spaces:
\[ H = \Lx2, \quad V_0 = H^1_0(\Omega), \quad V_0^* = (H^1_0(\Omega))^*, \quad V = \Hx1, \quad V^* = (\Hx1)^*, \]
\[ W_0 = \Hx2 \cap \Huo, \quad W = \left\{ u \in \Hx2 \mid \partial_{\n} u = 0 \right\}. \]
By standard results, we know that, if $H$ is identified with its dual, the following compact and dense embeddings hold:
\[ W_0 \hookrightarrow V_0 \hookrightarrow H \hookrightarrow V_0^* \quad \text{and} \quad W \hookrightarrow V \hookrightarrow H \hookrightarrow V^*.  \]
Additionally, we recall that, by elliptic regularity, we can use the equivalent norms:
\[ \norm{u}^2_{W_0} := \norm{u}^2_H + \norm{\Delta u}^2_H, \quad \norm{u}^2_W := \norm{u}^2_H + \norm{\Delta u}^2_H. \]
We further define the following spaces:
\[ \HH = H \times H \times H, \quad \VV = V_0 \times V \times V. \]
Regarding the parameters of the system, we assume the following hypotheses:
\begin{enumerate}[font = \bfseries, label = A\arabic*., ref = \bf{A\arabic*}]
	\item\label{ass:coeff} $\lambda, \eta, \gamma_h, \gamma_c, S_h, S_c, D, \gamma_p, \alpha_h, \alpha_c > 0$.
    \item\label{ass:Fh} $F(s) =  M s^2(1-s)^2$ and $\hh(s) = M s^2(3-2s)$, with $M>0$, for any $s \in \R$. In particular, we observe that $F, \hh \in \Ccal^\infty(\R)$. 
    \item\label{ass:m} $m(s) = m_{\text{ref}} \left( \frac{\rho + A}{2} + \frac{\rho - A}{2} \arctan \left( \frac{s - \sigma_l}{\sigma_r} \right) \right)$, with $m_{\text{ref}}, \rho, A, \sigma_l, \sigma_r > 0$. In particular, we observe that $m$ and $m'$ are Lipschitz continuous on $\R$ and $m, m', m'' \in \Lx \infty$.
\end{enumerate}
In what follows, we use the symbol $C \gs 0$, which may also change from line to line, to denote positive constants depending only on the fixed parameters of the system. In some cases, we will use a subscript to highlight some particular dependence of these constants.

In this section, we briefly recall the main analytical results proved in \cite{BCFLR2024}.
We refer the interested reader to the cited article for a more detailed exposition, as well as the proofs of the corresponding results.
We first introduce the set of admissible initial data as follows:
\begin{equation}
    \label{admset:initial}
    \begin{split}
    \Iad = & \big\{ (\phi_0, \sigma_0, p_0) \in \VV \mid 0 \le \phi_0 \le 1, \, 0 \le \sigma_0 \le \sigma_{\text{max}}, \\ 
    & \qquad \qquad 0 \le p_0 \le p_{\text{max}} \hbox{ and } \norm{(\phi_0, \sigma_0, p_0)}_{\VV} \le \Cb \big\},
    \end{split}
\end{equation}
where $\sigma_{\text{max}}, p_{\text{max}} \in \Lx \infty$ and $\Cb > 0$ are given. Note that $\Iad$ is a closed and convex subset of $\VV \cap \Lx\infty^3$.
Next, we consider the forward operator $\Rcal: \Iad \to \HH$ which associates to any initial data $(\phi_0, \sigma_0, p_0) \in \Iad$ the corresponding solution $(\phi(T), \sigma(T), p(T))$ to \eqref{eq:phi}--\eqref{ic}, evaluated at the final time, that is 
\begin{equation}
    \label{def:R}
    \Rcal: \Iad \to \HH, \quad \Rcal((\phi_0, \sigma_0, p_0)) = (\phi(T), \sigma(T), p(T)).
\end{equation}
By \cite[Theorem 3.2]{CGLMRR2019} and \cite[Proposition 2.3]{BCFLR2024}, we know that the forward map $\Rcal$ is well-defined and Lipschitz continuous.

\noindent
Moreover, by \cite[Theorem 2.5]{BCFLR2024}, $\Rcal$ is also continuously Fr\'echet differentiable in an open subset $\mathcal{I}_R$ of $\VV \cap \Lx\infty^3$, containing $\Iad$.
In particular, its derivative can be fully characterised in the following way. 
Indeed, we introduce the linearised system:
\begin{alignat}{2}
	& \partial_t Y - \lambda \Delta Y + F''(\phib) Y - m(\sigmab) \hh''(\phib) Y - m'(\sigmab) \hh'(\phib) Z = 0 
	\qquad && \hbox{in $Q_T$,} \label{eq:philin} \\
	& \partial_t Z - \eta \Delta Z + \gamma_h Z + (\gamma_c - \gamma_h) (\sigmab Y + \phib Z) - (S_c - S_h) Y = 0 
	\quad && \hbox{in $Q_T$,} \label{eq:sigmalin} \\
	& \partial_t P - D \Delta P + \gamma_p P 
	= (\alpha_c - \alpha_h) Y 
	&& \hbox{in $Q_T$,} \label{eq:plin} \\
	& Y = 0, \quad \partial_{\n} Z = \partial_{\n} P = 0 
	&& \hbox{on $\Sigma_T$,} \label{bclin} \\
	& Y(0) = h, \quad Z(0) = k, \quad P(0) = w
	&& \hbox{in $\Omega$,} \label{iclin}
\end{alignat}
where $h, k ,w \in \Lx2$. 
It is shown in \cite[Proposition 2.4]{BCFLR2024} that the linearised system \eqref{eq:philin}--\eqref{iclin} is well-posed, both in terms of weak and strong solutions. 


\noindent
Then, we can write the explicit expression of the Fr\'echet derivative of the operator $\Rcal$ as follows:
\begin{equation}
    \label{frechet:derivative}
    \D\Rcal(\phib_0, \sigmab_0, \pb_0)[(h,k,w)] = (Y(T), Z(T), P(T)).
\end{equation}
We stress that $\Rcal$ being of class $\Ccal^1$ means that the Fr\'echet derivative $\D \Rcal$ is Lipschitz continuous as a function from $\mathcal{I}_R$ to the space $\mathcal{L}(\VV \cap L^\infty(\Omega)^3, \HH)$. More precisely, the following estimate holds:
\begin{equation}
    \label{lipconst:frechet}
    \norm{\D\mathcal{R}((\phib_0^1, \sigmab_0^1, \pb_0^1)) - \D\mathcal{R}((\phib_0^2, \sigmab_0^2, \pb_0^2))}^2_{\mathcal{L}(\VV \cap \Lx\infty^3, \HH)} \le C_0 \norm{(\phib_0^1, \sigmab_0^1, \pb_0^1) - (\phib_0^2, \sigmab_0^2, \pb_0^2)}^2_{\HH},
\end{equation}
for some constant $C_0 \gs 0$ depending only on the parameters of the system.

The above regularity of the forward map opens the possibility to rigorously address the inverse problem of reconstructing the initial data, given a measurement at the final time.
Indeed, through a logarithmic convexity approach, in \cite{BCFLR2024} the authors proved two key stability results for the inverse problem, which we recall below. 
Call $M > 0$ and $M_1 > 0$ the minimal constants such that 
\begin{equation}
    \label{bound:c0h}
    \norm{(\phi, \sigma, p)}_{\C 0 {\HH}} \le M \quad \text{and} \quad \norm{(\phi, \sigma, p)}_{\HT1{\HH}} \le M_1
\end{equation}
uniformly for $(\phi_0, \sigma_0, p_0) \in \Iad$.
We recall that such estimates hold due to the well-posedness results in \cite[Theorem 3.2]{CGLMRR2019}.
Then, we have the following result.

\begin{theorem}
	\label{thm:inizstab}
	Assume hypotheses \ref{ass:coeff}--\ref{ass:m}. Let $(\phi_1, \sigma_1, p_1)$ and $(\phi_2, \sigma_2, p_2)$ be two solutions of \eqref{eq:phi}--\eqref{ic} corresponding to two triples of initial data $(\phi_0^i, \sigma_0^i, p_0^i) \in \Iad$ for $i=1,2$.
	Let $M \gs 0$, $M_1 \gs 0$ be as above.
	Moreover, let 
	\[ \eps := \frac{\norm{(\phi_1(T), \sigma_1(T), p_1(T)) - (\phi_2(T), \sigma_2(T), p_2(T))}_{\HH}}{M} \]
	and assume that 
	\[ \eps \le \exp \left\{ - \left( \min \left\{ 1, \frac{4\sqrt{3} M C_1^{3/2}}{9M_1} \right\} \right)^{-1} \right\}, \]
    where $C_1 = C_1(T) \gs 0$ is the constant appearing in the H\"older stability estimate of \cite[Proposition 3.1]{BCFLR2024}.
	Then, there exists a constant $C_2 \gs 0$ such that 
	\begin{equation}
		\label{eq:inizstab}
		\norm{ (\phi_0^1, \sigma_0^1, p_0^1) - (\phi_0^2, \sigma_0^2, p_0^2) }^2_{\HH} \le \frac{C_2}{\sqrt{\abs{\log \eps}}}, 
	\end{equation}
	where we can quantify the constant as 
	\[ C_2 = \frac{2M_1 M C_1^{1/2}}{\beta^{1/2}} + \frac{3M_1^2}{4\beta C_1}, \quad \text{with } \beta = \frac{\gamma}{e^{\gamma T}-1} \gs 0, \]
    where $\gamma \gs 0$ is a constant depending only on the parameters of the system, but not on $T$.
\end{theorem}

\begin{proof}
	See \cite[Theorem 3.5 and Proposition 3.1]{BCFLR2024}.
\end{proof}
\noindent
We observe that this type of logarithmic estimate \eqref{eq:inizstab} can be impractical in applications, as the error actually gets small only if the final data are very close to each other.
Moreover, it can be easily seen that the stability constant $C_2$ deteriorates exponentially with the final time $T$, making this stability estimate even more unreliable. 
Nevertheless, this is expected since we are dealing with a backward problem for a parabolic system, which is known to be severely ill-posed as $T$ gets larger.
However, in view of numerical applications, it is important to prove quantitative Lipschitz stability estimates, possibly under additional a priori hypotheses on the initial data.
Indeed, this was achieved as the main result of \cite{BCFLR2024}, by assuming to be reconstructing initial data lying in a finite-dimensional subspace of $\VV$, for instance, one of the discrete spaces used in numerical approximations.
We recall such a theorem below.

\begin{theorem}
	\label{thm:lipstab}
	Assume hypotheses \ref{ass:coeff}--\ref{ass:m}. Let $\Lambda$ be a finite-dimensional subspace of $\VV$ and $K \subseteq \Lambda$ be a compact subset.
    Assume further that 
	\[ (\phi_0, \sigma_0, p_0) \in K \cap \Iad. \]
	Then, there exists a constant $C_s \gs 0$ such that for any choice of $(\phi_0^1, \sigma_0^1, p_0^1)$ and $(\phi_0^2, \sigma_0^2, p_0^2)$ in $K \cap \Iad$ the following stability estimate holds
	\begin{equation}
		\label{eq:lipstab}
		\norm{ (\phi_0^1, \sigma_0^1, p_0^1) - (\phi_0^2, \sigma_0^2, p_0^2) }_{\HH} \le C_s \norm{ (\phi_1(T), \sigma_1(T), p_1(T)) - (\phi_2(T), \sigma_2(T), p_2(T)) }_{\HH},
	\end{equation}
	where the constant $C_s$ can be quantified as 
	\[ C_s = \max \left\{ \frac{2 \Cb}{M} e^{ \frac{16 C_0^2 C_2}{m_0}}, \frac{2}{m_0} \right\}, \text{ with } m_0 = \frac{L}{C_{\Lambda}} e^{-Q^2_2} \text{ and } C_{\Lambda} = \sup_{\vh \in \Lambda \setminus \{0\}} \frac{\norm{\vh}_{\VV}}{\norm{\vh}_{\HH}}, \]
	where $C_0$ is the Lipschitz constant of the Fr\'echet-derivative, given by \eqref{lipconst:frechet}, $L > 0$ is a uniform bound on the solutions to the linearised system and $Q_2 > 0$ is the logarithmic stability constant for the linearised system (see \cite[Theorem 3.8]{BCFLR2024}).
\end{theorem}
\begin{proof}
	See \cite[Theorem 3.10]{BCFLR2024}.
\end{proof}
\noindent
We finally point out that the constant $C_s$ blows up exponentially as the dimension of $\Lambda$ goes to infinity. 
Indeed, $C_s$ has a direct exponential proportionality to the constant $C_\Lambda$, which is finite only if $\Lambda$ has a finite dimension and blows up as it becomes higher. 
Additionally, one can also notice that the dependence of $C_s$ on the final time is even worse, namely $C_s$ depends doubly exponentially on $T$. 
This means that, even if we have some kind of Lipschitz stability, we have to be very careful when designing numerical algorithms to approximate the solution.

\section{Landweber iteration scheme}\label{sec:landweber}

Given a measurement $\triplemeas \in \HH$ at the final time, assuming the existence of a (unique) initial configuration $\tripleinbar$, we now address the problem of approximating this solution. 
By Theorem \ref{thm:lipstab}, we know that we have a Lipschitz stability estimate if we restrict the initial data in
\begin{equation}
	\label{def:K}
	K = \{ \vec{u} \in \Lambda \mid \norm{\vec{u}}_{\VV} \le \Cb \},
\end{equation}
where $\Lambda$ is a finite-dimensional subspace of $\VV$.  
Therefore, it makes sense to consider the minimisation problem:
\begin{equation}
	\label{eq:minproblem}
	\begin{split}
	& \argmin_{\triplein \in K \cap \Iad} \Jcal (\triplein) \\
	& \quad = \!\!\! \argmin_{\triplein \in K \cap \Iad} \frac{\kappa_1}{2} \norm{\phi(T) - \phimeas}^2_H + \frac{\kappa_2}{2} \norm{\sigma(T) - \sigmameas}^2_H + \frac{\kappa_3}{2} \norm{p(T) - \pmeas}^2_H,
	\end{split}
\end{equation}
where $\triplefin = \Rcal(\triplein)$ is the solution to \eqref{eq:phi}--\eqref{ic}, evaluated at the time $T$.  
Here $\kappa_1, \kappa_2, \kappa_3$ are positive parameters that can be chosen depending on which target is the most interesting in experiments and on the relative order of magnitude of the variables. 
Clearly, since $K \cap \Iad$ is compact and $\Jcal$ is continuous, \eqref{eq:minproblem} admits at least a solution in $K \cap \Iad$. Moreover, if a (unique) solution to the inverse problem exists in $K \cap \Iad$, then it trivially minimises $\Jcal$.
We recall that the finite-dimensional set $K$ can be chosen for instance as a closed and bounded subset of a discrete space used for numerical methods, therefore what we propose can be easily implemented.

\begin{remark}
    From a practical point of view, it could not be entirely feasible to have a pointwise value for the PSA concentration $p$, even if some ideas are given in \cite{Lorenzo2016}.
    Thus, a better term in the minimisation functional could be $\kappa_3 \left( \int_\Omega p \, \de x - P_{\text{meas}} \right)$, where $P_{\text{meas}}$ would be the global measured PSA level, which is what can be usually measured in practice.
    Of course, one can also consider this kind of term from a theoretical point of view and see that the minimisation of such a functional would lead to some optimality conditions that can then be discretised (cf. \cite{CGLMRR2021}).
    However, the corresponding iterative approximation method would not fit in the Landweber framework any longer. 
    Hence, we would not have at our disposal all the theoretical guarantees that are shown below for the Landweber scheme.
    For this reason, we stick to the proposed functional $\Jcal$, also because in practice (cf. Sections \ref{sec:numeth} and \ref{sec:sims}) we put $\kappa_2 = \kappa_ 3 = 0$, as the most important datum to reconstruct is the tumour phase field.
\end{remark}

As previously anticipated, to find such a solution, we use a Landweber iteration method, which is a common technique to approximate solutions to inverse problems. 
A similar idea was used for a simpler model of tumour growth in \cite{JBJA2019}. 
Indeed, starting from an initial guess $\tripleinin \in K \cap \Iad$, we approximate the solution $\tripleinbar$ by using a gradient descent algorithm of the following form:
\begin{equation}
	\label{eq:landweber}
	\vpsi_0^{j+1} = \vpsi_0^j - \mu \D\Rcal(\vpsi_0^j)^*[ \vec{\kappa} \cdot (\Rcal(\vpsi_0^j) - \vpsimeas) ], 
\end{equation}
where we used the compact notation $\vpsi_0^j = (\phi_0^j, \sigma_0^j, p_0^j)$, $\vpsimeas = \triplemeas$, $\vec{\kappa} = (\kappa_1, \kappa_2, \kappa_3)$ and $\D\Rcal(\cdot)^*[\cdot]$ is the adjoint operator of the Fr\'echet-derivative of $\Rcal$. 
The parameter $\mu$, instead, is a step size that has to be chosen suitably. 
In our case, it is easy to find the adjoint system associated with the minimisation problem \eqref{eq:minproblem}, for instance by using the formal Lagrangian method. 
Indeed, by also looking at \cite{CGLMRR2021}, the adjoint system associated to a generic state $\tripletil$ is:
\begin{alignat}{2}
	& - \partial_t q - \lambda \Delta q + F''(\phitil) q - m(\sigmatil) \hh''(\phitil) q \nonumber \\ 
    & \quad + (\gamma_c - \gamma_h) \sigmatil z - (S_c - S_h) z - (\alpha_c - \alpha_h) r = 0 
	\qquad && \hbox{in $Q_T$,} \label{eq:phiadj} \\
	& - \partial_t z - \eta \Delta z + \gamma_h z + (\gamma_c - \gamma_h) \phitil z - m'(\sigmatil) \hh'(\phitil) q = 0 
	\qquad && \hbox{in $Q_T$,} \label{eq:sigmaadj} \\
	& - \partial_t r - D \Delta r + \gamma_p r 
	= 0
	&& \hbox{in $Q_T$,} \label{eq:padj} \\
	& q = 0, \quad \partial_{\n} z = \partial_{\n} r = 0 
	&& \hbox{in $\Sigma_T$,} \label{bcadj} \\
	& q(T) = \kappa_1 (\phitil(T) - \phimeas), \, z(T) = \kappa_2 (\sigmatil(T) - \sigmameas), \nonumber \\
    & \quad r(T) = \kappa_3 (\ptil(T) - \pmeas)
	&& \hbox{in $\Omega$,} \label{icadj}
\end{alignat}
therefore we can say that 
\[ \D\Rcal^*(\vpsi_0^j)[ \vec{\kappa} \cdot (\Rcal(\vpsi_0^j) - \vpsimeas) ] = (q^j(0), z^j(0), r^j(0)), \]
where $(q^j,z^j,r^j)$ is the solution to the adjoint system \eqref{eq:phiadj}--\eqref{icadj} with $\tripletil$ solution corresponding to the initial data $(\phi_0^j, \sigma_0^j, p_0^j)$. Regarding the well-posedness of the adjoint system, we can state the following result.

\begin{proposition}
\label{prop:adjoint}
    Assume hypotheses \ref{ass:coeff}--\ref{ass:m} and let $\triplemeas \in \HH$. Assume further that $(\phitil_0, \sigmatil_0, \ptil_0) \in \Iad$ and that $\tripletil$ is the corresponding solution of \eqref{eq:phi}--\eqref{ic}. Then, \eqref{eq:phiadj}--\eqref{icadj} admits a unique weak solution $(q,z,r) \in \HT 1 {\VV^*} \cap \C 0 {\HH} \cap \LT 2 {\VV}$, which solves the system in variational formulation and satisfies the estimate:
    \begin{align*}
        & \norm{(q,z,r)}^2_{\HT 1 {\VV^*} \cap \C 0 {\HH} \cap \LT 2 {\VV}} \\
        & \quad \le C \left( \kappa_1^2 \norm{\phitil(T)-\phimeas}^2_H + \kappa_2^2 \norm{\sigmatil(T)-\sigmameas}^2_H + \kappa_3^2 \norm{\ptil(T)-\pmeas}^2_H \right),
    \end{align*}
    with $C > 0$ depending only on the parameters of the system.
\end{proposition}

\begin{proof}
    Using the transformation $t \mapsto T-t$, we can rewrite \eqref{eq:phiadj}--\eqref{icadj} as a linear parabolic system with initial conditions in $H$, coefficients in $\Lqt\infty$ (due to \cite[Theorem 3.2]{CGLMRR2019}) and without sources. Therefore, we can apply standard results on the existence of weak solutions of linear parabolic systems to deduce the thesis.  
\end{proof}
\noindent
We now establish convergence for the Landweber method and quantify its convergence rate, by using a general result proved in \cite[Theorem 3.2]{DQS2012}, which we recall below. 

\begin{lemma}
	\label{lem:scherzer}
	Let $X$ and $Y$ be two Hilbert spaces and $G: \mathcal{D}(G) \subseteq X \to Y$ be a continuous and locally Fr\'echet-differentiable operator. 
	Let $y \in Y$ and assume that there exists $\bar{x} \in \mathcal{D}(G)$ such that $G(\bar{x}) = y$.
	Consider the Landweber iteration scheme
	\[ x_{j+1} = x_j - \mu \D G(x_j)^*(G(x_j) - y), \]
	with starting point $x_0$. 
	Let $\rho \gs 0$ and assume that $\bar{x} \in B_{\rho}(x_0)$ and $\mathcal{B} = B_{\rho'}(x_0) \subseteq \mathcal{D}(G)$ for some $\rho' \gs \rho$.
	Moreover, let the following conditions be satisfied: 
	\begin{itemize}
		\item[$(i)$] The Fr\'echet-derivative $\D G$ of $G$ is Lipschitz continuous locally in $\mathcal{B}$, i.e.
		\[ \norm{\D G(x) - \D G(\tilde{x})}_{\mathcal{L}(X,Y)} \le L \norm{x - \tilde{x}}_X \quad \forall x, \tilde{x} \in \mathcal{B}. \]
		\item[$(ii)$] $G$ is weakly sequentially closed. 
		\item[$(iii)$] The inversion has a uniform Lipschitz-type stability, i.e.~there exists $C_G$ such that 
		\[ \norm{x - \tilde{x}}_X \le C_G \norm{G(x) - G(\tilde{x})}_Y \quad \forall x, \tilde{x} \in \mathcal{B}. \] 
		\item[$(iv)$] There exists $\hat{L} \gs 0$ such that $\norm{\D G(x)}_{\mathcal{L}(X,Y)} \le \hat{L}$ for any $x \in \mathcal{B}$.
		\item[$(v)$] The step-size $\mu$ is such that
        \[ \mu \ls \frac{1}{\hat{L}^2} \quad \text{and} \quad \mu (1 - \mu \hat{L}^2) \ls 2 C_G^2. \]
	\end{itemize}
	Then, if 
	\[ \rho = (2L\hat{L} C_G^2)^{-2} \quad \text{and} \quad \norm{x_0 - \bar{x}}_X^2 \le \rho, \]
	the iterates satisfy 
	\[ \norm{x_j - \bar{x}}_X^2 \le \rho \quad \forall x_j \in \N \quad \text{and} \quad x_j \to \bar{x} \text{ in $X$ as } j \to +\infty.  \]
	Moreover, if 
	\[ c = \mezzo \mu (1 - \mu \hat{L}^2) C_G^{-2}, \quad 0 \ls c \ls 1, \]
	the convergence rate is given by
    \begin{equation}
        \label{conv:rate}
        \norm{x_j - \bar{x}}^2_X \le \rho (1-c)^j.
    \end{equation}
\end{lemma}

\begin{remark}
	Observe that \eqref{conv:rate} means that at every step the error $\norm{x_j - \bar{x}}_X$ gets reduced by a factor $\sqrt{1-c}$. Then, the \emph{order} of convergence of the Landweber scheme as an iterative method is at most linear. 
\end{remark}
\noindent
We can now apply Lemma \ref{lem:scherzer} to our case, to prove convergence of the Landweber iteration \eqref{eq:landweber}. This should guarantee a fine numerical reconstruction of the initial distribution of the tumour, as long as the initial guess is chosen close enough to the actual solution.

\begin{theorem}
	\label{thm:landweber}
	Assume hypotheses \ref{ass:coeff}--\ref{ass:m}. Assume also that there exists $\bar{\vpsi}_0 \in K \cap \Iad$ such that $\Rcal(\bar{\vpsi}_0) = \vpsimeas$. Then, there exist $\rho \gs 0$, $0 \ls c \ls 1$ and $\mu^* \gs 0$, which depend only on the parameters of the system and can be determined as in Lemma $\ref{lem:scherzer}$, such that if $\vpsi^0_0 \in \HH$, $\vec{\kappa} \in \R^3$ and $\mu \gs 0$ are chosen such that 
	\[ \norm{ \vpsi^0_0 - \bar{\vpsi}_0 }^2_{\HH} \le \rho \quad \text{and} \quad \mu \abs{\vec{\kappa}} \ls \mu^*, \]
	then the iterates $\{ \vpsi_0^j \}_{j \in\N}$ of \eqref{eq:landweber} satisfy
	\[ \norm{\vpsi_0^j - \bar{\vpsi}_0}^2_{\HH}  \le \rho (1-c)^j \text{ for any $j\in\N$ and } \vpsi_0^j \to \bar{\vpsi}_0 \text{ as $j \to +\infty$}. \]
\end{theorem}


\begin{proof}
	We just need to apply Lemma \ref{lem:scherzer} with $X = \Lambda$, $Y = \HH$ and $G = \Rcal$. Indeed, hypothesis $(i)$ on Lipschitz-continuity of the Fr\'echet-derivative $\D \Rcal$ is verified even globally by the Fr\'echet-differentiability result in \cite[Theorem 2.5]{BCFLR2024}. Then, condition $(ii)$ on weak sequential closedness follows immediately by the well-posedness result in \cite[Theorem 3.2]{CGLMRR2019}, which implies continuity of $\Rcal$, and by the fact that $\Lambda$ is finite-dimensional, so weak and strong topologies coincide. Moreover, hypothesis $(iii)$ is exactly Theorem \ref{thm:lipstab}, giving Lipschitz stability of the inverse map, and condition $(iv)$ follows easily by \cite[Proposition 2.4]{BCFLR2024} on well-posedness of the linearised system. 
    Regarding the correspondence of the constants appearing in Lemma \ref{lem:scherzer}, we have that in our case $L = C_0$ given by \eqref{lipconst:frechet}, $C_G = C_s$ given by \eqref{eq:lipstab} and $\hat{L}$ is related to the well-posedness of the linearised system \eqref{eq:philin}--\eqref{iclin}.
    In particular, $\hat{L}$ should be such that
    \[ \norm{\D \Rcal(\phi_0, \sigma_0, p_0)}_{\mathcal{L}(\Lambda, \HH)} = \sup_{\norm{(h,k,w)}_{\HH} = 1} \, \norm{(Y(T), Z(T), P(T))}_{\HH} \le \hat{L}, \]
    for any choice of $(\phi_0, \sigma_0, p_0) \in \Lambda$.
    Additionally, we recall that, in our formulation \eqref{eq:landweber} of the algorithm, we also added the weights $\vec{\kappa}$, which can be also factored out of the adjoint operator, since it is linear. This means that we actually have to impose a smallness condition like $(v)$ above on $\mu \abs{\vec{\kappa}}$, i.e. 
    \[ \mu \abs{\vec{\kappa}} \ls \frac{1}{\hat{L}^2} \quad \text{and} \quad \mu \abs{\vec{\kappa}} (1 - \mu \abs{\vec{\kappa}} \hat{L}^2) \ls 2 C_s^2. \]
    Then $\mu^*$ can be deduced by the previous inequalities and we can define 
    \[ \rho = (2C_0\hat{L} C_s^2)^{-2} \quad \text{and} \quad c = \mezzo \mu \abs{\vec{\kappa}}(1 - \mu \abs{\vec{\kappa}} \hat{L}^2) C_s^{-2}. \]
    Therefore Lemma \ref{lem:scherzer} can be applied, and the proof is concluded.
\end{proof}

\begin{remark}
	Observe that, by Theorem \ref{thm:landweber}, the scheme \eqref{eq:landweber} converges if the initial guess $\vpsi_0^0$ is chosen such that $\norm{\vpsi^0_0 - \bar{\vpsi}_0}^2_{\HH} < \rho$. However, theoretically, the Lipschitz stability estimate \eqref{eq:lipstab} gives us an explicit condition to choose our starting guess. Indeed, by Theorem \ref{thm:lipstab}, we know that, if we assume $\bar{\vpsi}_0 \in K \cap \Iad$, we can expect that 
    \[ \norm{\vpsi^0_0 - \bar{\vpsi}_0}_{\HH} \le C_s \norm{ \Rcal(\vpsi_0^0) - \vpsimeas }_{\HH}, \]
    where the term on the right-hand side is explicitely computable in practice.
    Then, with a similar reasoning as in \cite[Lemma 2]{AS2022}, we just need to find $\vpsi_0^0 \in K$ such that 
	\begin{equation}
		\label{eq:guess}
		\norm{ \Rcal(\vpsi_0^0) - \vpsimeas }_{\HH} \le \frac{\sqrt{\rho}}{C_s}.
	\end{equation}
    This would immediately imply that $\norm{\vpsi^0_0 - \bar{\vpsi}_0}^2_{\HH} < \rho$.
	Since $K$ is compact, \eqref{eq:guess} can be guaranteed for example by covering $K$ with a fine enough finite lattice or by random sampling until one finds a starting point close enough.
    Clearly, this is not efficient and often unpractical when working with high-dimensional discretisations and long final times $T$, since, by its definition, the bound in \eqref{eq:guess} decays at least exponentially with both dimension and final time.  
    Due to these reasons, for our experiments in Section \ref{sec:sims}, we choose the initial guess through physical considerations. 
\end{remark}

While being an interesting result in theory, Theorem \ref{thm:landweber} requires highly restrictive hypotheses to be verified in practice. In particular, the convergence radius $\rho$ becomes almost immediately too small when increasing the final time or the dimension of the discrete spaces. Moreover, the bound on the constant step-size $\mu$ is not easily computable. To overcome these issues, one usually tries to set up an adaptive choice of the step-size, thus considering algorithms of the form: 
\begin{equation}
	\label{eq:landweber_adapt}
	\vpsi_0^{j+1} = \vpsi_0^j - \mu^j \D\Rcal(\vpsi_0^j)^*[ \vec{\kappa} \cdot (\Rcal(\vpsi_0^j) - \vpsimeas) ]. 
\end{equation}
There are many possible choices of the adaptive step-size, due to the extensive literature on general gradient descent methods. One of the most commonly used is the \emph{steepest descent}, which essentially amounts to choosing the step that gives the biggest reduction along the direction of the gradient. Following \cite[Section 3.4]{kaltenbacher:neubauer:scherzer}, in case of our Landweber iteration this choice becomes
\begin{equation}
    \label{adaptive:step}
    \mu^j \abs{\vec{\kappa}} = \frac{\norm{\D\Rcal(\vpsi_0^j)^*[ \vec{\kappa} \cdot (\Rcal(\vpsi_0^j) - \vpsimeas) ]}^2_{\HH}}{\norm{\D\Rcal(\vpsi_0^j)[\D\Rcal(\vpsi_0^j)^*[ \vec{\kappa} \cdot (\Rcal(\vpsi_0^j) - \vpsimeas) ]]}^2_{\HH}} = \frac{\norm{(q^j(0), z^j(0), r^j(0))}^2_{\HH}}{\norm{(Y^j(T), Z^j(T), P^j(T))}^2_{\HH}},
\end{equation}
where $(q^j,z^j,r^j)$ is the solution to the adjoint system \eqref{eq:phiadj}--\eqref{icadj} with $\tripletil = (\phi^j, \sigma^j, p^j)$ solution corresponding to the initial data $(\phi_0^j, \sigma_0^j, p_0^j)$ as before, and $(Y,Z,P)$ is the solution to the linearised system \eqref{eq:philin}--\eqref{iclin} with initial data $(h,k,w) = (q^j(0), z^j(0), r^j(0))$ and $(\phib, \sigmab, \pb) = (\phi^j, \sigma^j, p^j)$ solution corresponding to the initial data $(\phi_0^j, \sigma_0^j, p_0^j)$, namely 
\begin{alignat*}{2}
	& \partial_t Y - \lambda \Delta Y + F''(\phi^j) Y - m(\sigma^j) \hh''(\phi^j) Y - m'(\sigma^j) \hh'(\phi^j) Z = 0 
	\qquad && \hbox{in $Q_T$,} \\
	& \partial_t Z - \eta \Delta Z + \gamma_h Z + (\gamma_c - \gamma_h) (\sigma^j Y + \phi^j Z) - (S_c - S_h) Y = 0 
	\quad && \hbox{in $Q_T$,} \\
	& \partial_t P - D \Delta P + \gamma_p P 
	= (\alpha_c - \alpha_h) Y 
	&& \hbox{in $Q_T$,} \\
	& Y = 0, \quad \partial_{\n} Z = \partial_{\n} P = 0 
	&& \hbox{on $\Sigma_T$,} \\
	& Y(0) = q^j(0), \quad Z(0) = z^j(0), \quad P(0) = r^j(0)
	&& \hbox{in $\Omega$.}
\end{alignat*}
It is shown in \cite[Theorems 3.21 and 3.22]{kaltenbacher:neubauer:scherzer} that this method converges to the exact solution $\tripleinbar$, under the validity of the tangential cone condition on $\Rcal$ and the choice of a suitable stopping criterion for the number of iterations, via a discrepancy principle. For the sake of completeness, we recall that the operator $\Rcal$ satisfies the tangential cone condition if there exist $r \gs 0$ and $\eta \ls \mezzo$ such that
\begin{equation}
    \label{eq:tangentialcone}
    \norm{\Rcal(\vpsi_0) - \Rcal(\widetilde{\vpsi}_0) - \D \Rcal(\vpsi_0)[\vpsi_0 - \widetilde{\vpsi}_0]}_{\HH} \le \eta \norm{\Rcal(\vpsi_0) - \Rcal(\widetilde{\vpsi}_0)}_{\HH} \quad \forall \vpsi_0, \widetilde{\vpsi}_0 \in B_{2r}(\vpsi_0^0).
\end{equation}
In our case, since $\Rcal$ is Fr\'echet-differentiable with a Lipschitz continuous derivative and the stability estimate \eqref{eq:lipstab} holds, it is easy to see that this local condition holds. Indeed, by straight-forward computations, it follows that 
\begin{align*}
   & \norm{\Rcal(\vpsi_0) - \Rcal(\widetilde{\vpsi}_0) - \D \Rcal(\vpsi_0)[\vpsi_0 - \widetilde{\vpsi}_0]}_{\HH} \le C_0 \norm{\vpsi_0 - \widetilde{\vpsi}_0}^2_{\HH} \\
   & \quad = C_0 \norm{\vpsi_0 - \widetilde{\vpsi}_0}_{\HH} \norm{\vpsi_0 - \widetilde{\vpsi}_0}_{\HH} \le C_0 C_s 4 r \,\norm{\Rcal(\vpsi_0) - \Rcal(\widetilde{\vpsi}_0)}_{\HH},
\end{align*}
where $\eta = C_0 C_s 4 r \ls \mezzo$, if $r \gs 0$ is small enough. As said above (cf. \cite{kaltenbacher:neubauer:scherzer}), this condition guarantees local convergence of the adaptive Landweber method \eqref{eq:landweber_adapt}. 
We further comment that the Landweber method, both with the fixed step size and the steepest descent version, is known to be a reliable regularisation method in the presence of noisy data (cf. \cite{kaltenbacher:neubauer:scherzer}).
Indeed, if paired with a suitable stopping rule for the number of iterations, the method reconstructs a regularised version of the initial data and converges to the exact solution if the noise level goes to zero. 
The reason why the number of iterations has to be stopped accordingly in the presence of noise is to avoid the amplification of errors, as the Landweber scheme shows a typical semi-convergence behaviour.
This will be clearer by looking at the simulation results in Section \ref{sec:sims}.

However, due to its slow convergence, the proposed Landweber schemes may still be not enough to have convergence when the final time $T$ is too large, since all the stability constants blow up exponentially.
For this reason, when doing numerical simulations, we are forced to adopt more efficient adaptive step-size choices than those mentioned above. 
The literature on methods for accelerating Landweber iterations or, more in general, gradient descent schemes is definitely vast, we cite for example \cite{HR2017, BBT2017, DHS2011, nesterov} and references therein. 
In particular, upon many trial and error procedures, we especially mention \cite{MM2019adaptive}, since the method proposed there proved particularly useful to our situation. 
Indeed, they propose a new choice of the adaptive step size for a general gradient descent method and they show its convergence, both theoretically and numerically, for a convex function in $\R^n$ with locally Lipschitz gradient.
Their proposed choice for the step size is particularly easy to implement and in our case takes the following form
\begin{align}
    & \mu^j \abs{\vec{\kappa}} = \min \left\{ \sqrt{1 + \theta^{j-1}} \, \mu^{j-1} \abs{\vec{\kappa}}, \frac{\norm{\vpsi_0^j - \vpsi_0^{j-1}}_{\HH}}{2 \norm{((q^j - q^{j-1})(0), (z^j - z^{j-1})(0), (r^j - r^{j-1})(0)) }_{\HH}} \right\}, \nonumber \\
    & \hbox{where } \theta^{j} = \frac{\mu^j}{\mu^{j-1}} \hbox{ for any $j \in \N$ and $\mu^0 \gs 0$, $\theta^0 = + \infty$ are given as inputs.} \label{eq:malitsky}
\end{align}
Heuristically, the first term inside the minimum assures that some kind of discrete Lyapunov energy associated to the scheme keeps decreasing, while the second one is an approximation of the reciprocal of the Lipschitz constant of the descent direction.
Despite the lack of analysis in infinite-dimensional spaces with non-convex operators, this method still proved powerful in tackling our problem.
More details will be given when speaking about numerical results in the next sections. 
To distinguish the two proposed variants of the choice of the adaptive step size, in what follows we will denote \eqref{eq:landweber_adapt} with the steepest descent choice \eqref{adaptive:step} as the Landweber scheme and \eqref{eq:landweber_adapt} with the second choice \eqref{eq:malitsky} as the Adaptive Gradient Descent scheme. 
Keep in mind, however, that both schemes fit into the Landweber framework, even if we call them differently. 

\section{Numerical methods}\label{sec:numeth}

\subsection{Spatial discretisation}\label{sec:sdisc}
We use Isogeometric Analysis (IGA) to discretise in space the forward, linearised, and adjoint problems. IGA is a recent, rapidly-growing technique that can be regarded as a generalisation of the classical Finite Element Method \cite{Cottrell2009}. In particular, we use a standard isogeometric Bubnov-Galerkin method based on a $C^1$-continuous quadratic B-spline space (see Ref.~\cite{Cottrell2009} for further detail). Hence, our spatial discretisation requires the definition of the weak forms of the forward, linearised, and adjoint problems. Let us first consider the forward problem in Eqs.~\eqref{eq:phi}-\eqref{eq:p}, whose weak form can be written as: find $\phi\in V_0$, $\sigma\in V$, and $p\in V$ such that
\begin{align} 
\label{bwf1} & B_1^\chi(\chi_1,\phi,\sigma,p)=0\quad\text{for all}\quad \chi_1\in V_0, \\
\label{bwf2} & B_2^\chi(\chi_2,\phi,\sigma,p)=0\quad\text{for all}\quad \chi_2\in V,   \\
\label{bwf3} & B_3^\chi(\chi_3,\phi,\sigma,p)=0\quad\text{for all}\quad \chi_3\in V, 
\end{align}
\noindent where
\begin{align} 
\label{wf1} B_1^\chi(\chi_1,\phi,\sigma,p) & = \int_{\Omega} \chi_1 \left[ \partial_t \phi+ F^\prime(\phi) - m(\sigma)\hh^\prime(\phi) \right] \, \de x  + \int_{\Omega} \lambda\nabla \chi_1\cdot  \nabla \phi \, \de x, \\
\label{wf2} B_2^\chi(\chi_2,\phi,\sigma,p) & = \int_{\Omega} \chi_2 \left[ \partial_t \sigma + \gamma_h\sigma + \gamma_{ch}\sigma \phi  - S_h - S_{ch} \phi \right] \,\de x +  \int_{\Omega} \eta\nabla \chi_2\cdot  \nabla \sigma \,\de x, \\
\label{wf3} B_3^\chi(\chi_3,\phi,\sigma,p) & = \int_{\Omega} \chi_3 \left[ \partial_t p + \gamma_p p -\alpha_h - \alpha_{ch}\phi \right] \,\de x  + \int_{\Omega} D\nabla \chi_3\cdot  \nabla p \,\de x.
\end{align}
We also introduce the weak form of the linearised problem defined by Eqs.~\eqref{eq:philin}-\eqref{eq:plin}, which is stated as: find $Y\in V_0$, $Z\in V$, and $P\in V$ such that
\begin{align} 
\label{b2wf1} & B_1^\zeta(\zeta_1,Y,Z,P)=0\quad\text{for all}\quad \zeta_1\in V_0, \\
\label{b2wf2} & B_2^\zeta(\zeta_2,Y,Z,P)=0\quad\text{for all}\quad \zeta_2\in V,   \\
\label{b2wf3} & B_3^\zeta(\zeta_3,Y,Z,P)=0\quad\text{for all}\quad \zeta_3\in V, 
\end{align}
\noindent where
\begin{align} 
\label{lwf1} B_1^\zeta(\zeta_1,Y,Z,P) & = \int_{\Omega} \zeta_1 \left[ \partial_t Y + F^{\prime\prime}(\bar{\phi})Y - m(\bar{\sigma})\hh^{\prime\prime}(\bar{\phi})Y - m^\prime(\bar{\sigma})\hh^\prime({\bar{\phi}})Z \right]\, \de x \\
& \notag \quad + \int_{\Omega} \lambda\nabla \zeta_1\cdot \nabla Y \de x, \\
\label{lwf2} B_2^\zeta(\zeta_2,Y,Z,P) & = \int_{\Omega} \zeta_2 \left[ \partial_t Z + \gamma_h Z + \gamma_{ch}(\bar{\sigma}Y  + \bar{\phi}Z) - S_{ch}Y \right] \,\de x +  \int_{\Omega} \eta\nabla \zeta_2\cdot\nabla Z \,\de x, \\
\label{lwf3} B_3^\zeta(\zeta_3,Y,Z,P) & = \int_{\Omega} \zeta_3 \left[ \partial_t P + \gamma_p P - \alpha_{ch} Y \right] \,\de x  + \int_{\Omega} D\nabla \zeta_3\cdot  \nabla P \,\de x.
\end{align}
Likewise, we express the weak form of the adjoint problem given by Eqs.~\eqref{eq:phiadj}-\eqref{eq:padj} as follows: find $q\in V_0$, $z\in V$, and $r\in V$ such that
\begin{align} 
\label{b3wf1} & B_1^\psi(\psi_1,q,z,r)=0\quad\text{for all}\quad \psi_1\in V_0, \\
\label{b3wf2} & B_2^\psi(\psi_2,q,z,r)=0\quad\text{for all}\quad \psi_2\in V,   \\
\label{b3wf3} & B_3^\psi(\psi_3,q,z,r)=0\quad\text{for all}\quad \psi_3\in V, 
\end{align}
\noindent where
\begin{align} 
\label{awf1} B_1^\psi(\psi_1,q,z,r) & = \int_{\Omega} \psi_1 \left[ -\partial_t q + F^{\prime\prime}(\bar{\phi})q - m(\bar{\sigma})\hh^{\prime\prime}(\bar{\phi})q + \gamma_{ch}\bar{\sigma}z - S_{ch}z - \alpha_{ch}r \right] \, \de x \\
& \notag \quad + \int_{\Omega} \lambda\nabla \psi_1\cdot \nabla q \, \de x, \\
\label{awf2} B_2^\psi(\psi_2,q,z,r) & = \int_{\Omega} \psi_2 \left[ -\partial_t z + \gamma_h z + \gamma_{ch}\bar{\phi}z  - m^\prime(\bar{\sigma})\hh^\prime(\bar{\phi})q \right] \, \de x +  \int_{\Omega} \eta\nabla \psi_2\cdot\nabla z \, \de x, \\
\label{awf3} B_3^\psi(\psi_3,q,z,r) & = \int_{\Omega} \psi_3 \left[ -\partial_t r + \gamma_p r \right] \, \de x  + \int_{\Omega} D\nabla \psi_3\cdot  \nabla r \, \de x.
\end{align}
The spatial discretisation of the weak forms defined above further relies on defining finite-dimensional spaces $V^h\subset V$ and $V_0^h\subset V_0$. We construct these discrete spaces leveraging the aforementioned $C^1$-continuous quadratic B-spline space \cite{Cottrell2009}. For example, the space $V^h$ can be defined in terms of a spline basis as $V^h=\mbox{span}\left\{ N_A(x)\right\}_{A=1,\dots,n_f}$, where $n_f$ is the number of basis functions (i.e., $n_f=\dim(V^h)$) and $N_A(x)$ represents each multivariate spline basis function. We also utilise the superscript $h$ to denote finite-dimensional approximations to the exact solution of the forward, linearised, and adjoint problems. For instance, the finite-dimensional approximation to the tumour phase field is given by $\phi^h(t,x)=\sum_{A=1}^{n_f}\phi_A(t)N_A(x)$, where the time-dependent coefficients $\phi_A(t)$ are called control variables. The functions $\sigma^h$, $p^h$, $Y^h$, $Z^h$, $P^h$, $q^h$, $z^h$, and $r^h$ are defined analogously. Importantly, the functions that belong to $V_0^h$ will have some control variables constrained to ensure that the Dirichlet boundary conditions are satisfied. Furthermore, the Neumann boundary conditions that are relevant to the forward, linearised, and adjoint problems are naturally enforced within the weak form.
We finally mention that the space $\VV^h := V_0^h \times V^h \times V^h$ can then be considered as an example of the finite-dimensional subspace $\Lambda \subseteq \VV$ in Theorem \ref{thm:lipstab}. Actually, in the following, we will tacitly assume $\Lambda =\VV^h$.

\subsection{Time discretisation}\label{sec:tdisc}
We use the generalized-$\alpha$ method to integrate in time \cite{Cottrell2009,Chung1993,Jansen2000}. This method is applied to the spatially-discretised version of the weak forms of the forward, linearised, and adjoint problems introduced in Section~\ref{sec:sdisc}. Let us first consider the forward problem. We define $\boldsymbol{\phi}$ as the global vector of control variables associated with the unknown field $\phi^h$, i.e., $\boldsymbol{\phi}=\left\{\phi_A \right\}_{A=1,\dots,n_f}$. Similarly, we further define the vectors $\boldsymbol{\sigma}=\left\{\sigma_A \right\}_{A=1,\dots,n_f}$ and  $\boldsymbol{p}=\left\{p_A \right\}_{A=1,\dots,n_f}$. We can now introduce the residual vector of the forward problem as
\begin{equation}
\vec{\rm Res^F} = \{ \boldsymbol{R}^\phi,\boldsymbol{R}^\sigma, \boldsymbol{R}^p\} 
\end{equation}
where $\boldsymbol{R}^\phi=\{R^\phi_A\}_{A=1,\dots,n_f}$, $\boldsymbol{R}^\sigma=\{R^\sigma_A\}_{A=1,\dots,n_f}$, and $\boldsymbol{R}^p=\{R^p_A\}_{A=1,\dots,n_f}$, such that
\begin{align} 
R^\phi_A   & = B_1^\chi(N_A,\boldsymbol{\phi},\boldsymbol{\sigma},\boldsymbol{p}),\\
R^\sigma_A & = B_2^\chi(N_A,\boldsymbol{\phi},\boldsymbol{\sigma},\boldsymbol{p}),\\
R^p_A      & = B_3^\chi(N_A,\boldsymbol{\phi},\boldsymbol{\sigma},\boldsymbol{p}).
\end{align}
Let us further define $\boldsymbol{U}_n=\{\boldsymbol{\phi}_n,\boldsymbol{\sigma}_n,\boldsymbol{p}_n\}$ as the time-discrete approximation to the control variables of the forward problem at time $t_n$.
Then, in the forward problem we calculate $\boldsymbol{U}_{n+1}$ from $\boldsymbol{U}_n$ by enforcing the equation
\begin{equation}\label{resi}
\vec{\rm Res^F}(\dot{\boldsymbol{U}}_{n+\alpha_m},\boldsymbol{U}_{n+\alpha_f})=\vec 0,
\end{equation}
where
\begin{alignat}{2}
&\vec U_{n+1}              &= &\,\vec U_n+(t_{n+1}-t_n)\dot{\vec U}_n+\gamma(t_{n+1}-t_n)(\dot{\vec U}_{n+1}-\dot{\vec U}_n), \label{resii} \\
&\dot{\vec U}_{n+\alpha_m} &= &\,\dot{\vec U}_n+\alpha_m(\dot{\vec U}_{n+1}-\dot{\vec U}_n),  \\
&{\vec U}_{n+\alpha_f}     &= &\,{\vec U}_n+\alpha_f({\vec U}_{n+1}-{\vec U}_n). \label{resf}
\end{alignat}
In Eqs.~\eqref{resi}-\eqref{resf}, we define $t_{n+1}-t_n=\Delta t_n>0$ as the time step, while $\alpha_m$, $\alpha_f$ and $\gamma$ are real-valued parameters that define the accuracy and stability of the time integration algorithm.

Similarly, for the linearised problem, we introduce the global vectors $\boldsymbol{Y}=\left\{Y_A \right\}_{A=1,\dots,n_f}$, $\boldsymbol{Z}=\left\{Z_A \right\}_{A=1,\dots,n_f}$ and $\boldsymbol{P}=\left\{P_A \right\}_{A=1,\dots,n_f}$ corresponding to the discrete functions $Y^h$, $Z^h$, and $P^h$. We further define the residual vector for the linearised problem as
\begin{equation}
\vec{\rm Res^L} = \{ \boldsymbol{R}^Y,\boldsymbol{R}^Z, \boldsymbol{R}^P\} 
\end{equation}
where $\boldsymbol{R}^Y=\{R^Y_A\}_{A=1,\dots,n_f}$, $\boldsymbol{R}^Z=\{R^Z_A\}_{A=1,\dots,n_f}$ and $\boldsymbol{R}^P=\{R^P_A\}_{A=1,\dots,n_f}$, such that
\begin{align} 
R^Y_A   & = B_1^\zeta(N_A,\boldsymbol{Y},\boldsymbol{Z},\boldsymbol{P}),\\
R^Z_A   & = B_2^\zeta(N_A,\boldsymbol{Y},\boldsymbol{Z},\boldsymbol{P}),\\
R^P_A   & = B_3^\zeta(N_A,\boldsymbol{Y},\boldsymbol{Z},\boldsymbol{P}).
\end{align}
Now, we can redefine $\boldsymbol{U}_n=\{\boldsymbol{Y}_n,\boldsymbol{Z}_n,\boldsymbol{P}_n\}$ using the time-discrete approximation to the control variables of the linearised problem at time $t_n$. Then, the time integration of the linearised problem consists of calculating $\boldsymbol{U}_{n+1}$ from $\boldsymbol{U}_n$ by requiring
\begin{equation}\label{resl}
\vec{\rm Res^L}(\dot{\boldsymbol{U}}_{n+\alpha_m},\boldsymbol{U}_{n+\alpha_f})=\vec 0,
\end{equation}
and using the expressions for $\boldsymbol{U}_{n+1}$, $\boldsymbol{U}_{n+\alpha_m}$, and $\boldsymbol{U}_{n+\alpha_f}$ provided by Eqs.~\eqref{resii}-\eqref{resf}. 

For the adjoint problem, we make the same initial definitions as we did for the forward and linearised problem. Hence, we introduce the global vectors $\boldsymbol{q}=\left\{q_A \right\}_{A=1,\dots,n_f}$, $\boldsymbol{z}=\left\{z_A \right\}_{A=1,\dots,n_f}$ and $\boldsymbol{r}=\left\{r_A \right\}_{A=1,\dots,n_f}$, which correspond to the discrete functions $q^h$, $z^h$, and $r^h$. Furthermore, we define the residual vector for the adjoint problem as
\begin{equation}
\vec{\rm Res^A} = \{ \boldsymbol{R}^q,\boldsymbol{R}^z, \boldsymbol{R}^r\} 
\end{equation}
where $\boldsymbol{R}^q=\{R^q_A\}_{A=1,\dots,n_f}$, $\boldsymbol{R}^z=\{R^z_A\}_{A=1,\dots,n_f}$, and $\boldsymbol{R}^r=\{R^r_A\}_{A=1,\dots,n_f}$, such that
\begin{align} 
R^q_A   & = B_1^\psi(N_A,\boldsymbol{q},\boldsymbol{z},\boldsymbol{r}),\\
R^z_A   & = B_2^\psi(N_A,\boldsymbol{q},\boldsymbol{z},\boldsymbol{r}),\\
R^r_A   & = B_3^\psi(N_A,\boldsymbol{q},\boldsymbol{z},\boldsymbol{r}).
\end{align}
Let us now redefine $\boldsymbol{U}_n=\{\boldsymbol{q}_n,\boldsymbol{z}_n,\boldsymbol{r}_n\}$ using the time-discrete approximation to the control variables of the adjoint problem at time $t_n$. A fundamental difference of the adjoint problem with respect to the forward and linearised problems is that the former is solved backwards in time starting at $t=T$, so we need to redefine $\Delta t_n=t_n-t_{n+1}<0$. Then, the time integration of the adjoint problem consists of computing $\boldsymbol{U}_n$ from $\boldsymbol{U}_{n+1}$ by imposing 
\begin{equation}\label{resaa}
\vec{\rm Res^A}(\dot{\boldsymbol{U}}_{n+\alpha_m},\boldsymbol{U}_{n+\alpha_f})=\vec 0,
\end{equation}
where
\begin{alignat}{2}
&\vec U_{n}              &= &\,\vec U_{n+1}+(t_{n}-t_{n+1})\dot{\vec U}_{n+1}+\gamma(t_{n}-t_{n+1})(\dot{\vec U}_{n}-\dot{\vec U}_{n+1}), \label{resai} \\
&\dot{\vec U}_{n+\alpha_m} &= &\,\dot{\vec U}_{n+1}+\alpha_m(\dot{\vec U}_{n}-\dot{\vec U}_{n+1}),  \\
&{\vec U}_{n+\alpha_f}     &= &\,{\vec U}_{n+1}+\alpha_f({\vec U}_{n}-{\vec U}_{n+1}). \label{resaf}
\end{alignat}
Note that we have redefined $\boldsymbol{U}_{n+1}$, $\boldsymbol{U}_{n+\alpha_m}$, and $\boldsymbol{U}_{n+\alpha_f}$ in Eqs.~\eqref{resai}-\eqref{resaf} with respect to their former expressions used for the time integration of the forward and linearised problems provided by Eqs.~\eqref{resii}-\eqref{resf} in order to accommodate the backwards nature of the adjoint problem.

As shown in Ref.~\cite{Jansen2000}, $A$-stability and second-order accuracy can be obtained with the generalized-$\alpha$ method if $\rho_\infty\in[0,1]$ and
\begin{equation}\label{alphas}
\alpha_m=\frac{1}{2}\left(\frac{3-\rho_\infty}{1+\rho_\infty}\right),\quad \alpha_f=\frac{1}{1+\rho_\infty},\quad \gamma=\frac{1}{2}+\alpha_m-\alpha_f.
\end{equation}
All the simulations presented in this work were carried out by using $\rho_\infty=1/2$ and leveraging the definitions provided in Eq.~\eqref{alphas}. Additionally, we used the Newton-Raphson method \cite{Cottrell2009} to solve the algebraic systems given by Eqs.~\eqref{resi}-\eqref{resf} for the forward problem, Eqs.~\eqref{resl} and \eqref{resii}-\eqref{resf} for the linearised problem, and Eqs.~\eqref{resaa}-\eqref{resaf} for the adjoint problem. For each problem, the convergence criterion to advance from one time step to the next one consists of reducing the individual residuals to $\eps_{NR}$ of its initial value (i.e., $\vec R^\phi$, $\vec R^\sigma$, and $\vec R^p$ for the forward problem; $\vec R^Y$, $\vec R^Z$, and $\vec R^P$ for the linearised problem; and $\vec R^q$, $\vec R^z$, and $\vec R^r$ for the adjoint problem). Then, the linear systems that result after the application of the Newton-Raphson method are solved using GMRES \cite{Saad1986} with a diagonal preconditioner up to a predefined tolerance $\eps_{GMRES}$ or a maximum number of iterations.

\subsection{Iterative algorithms for the identification of initial data}
In the simulation study conducted in this work (see Section~\ref{sec:sims}), we will focus on reconstructing the initial tumour phase field $\phi_0$.
The rationale for this choice is that knowledge on the tumour burden (e.g., volume, extension) is a central piece of information in clinical decision-making for prostate cancer management \cite{Mottet2021,Cornford2021}.
Additionally, clinical imaging techniques that are currently used to diagnose, monitor and plan treatments for prostate cancer patients could be leveraged to obtain measurements of the spatial map of the tumour phase field (e.g., multiparametric magnetic resonance imaging) \cite{Mottet2021,Cornford2021,Lorenzo2022_review,Lorenzo2024}, including $\phi_\mathrm{meas}$ at time $t=T$ for the purposes of the present work.
Then, given a reconstruction of $\phi_0$, for every step we can approximate the initial data $\sigma_0$ and $p_0$ by using linear phenomenological laws, as already done in \cite{CGLMRR2019, CGLMRR2021}:
\begin{equation}\label{eq:s0p0}
    \sigma_0 = c_{0,\sigma} + c_{1,\sigma} \phi_0, \qquad p_0 = c_{0,p} + c_{1,p} \phi_0,
\end{equation}
where $c_{0,\sigma}, c_{1,\sigma}, c_{0,p}, c_{1,p} \in \R$ are explicit coefficients. 
In this context, we set $\kappa_1 = 1$, $\kappa_2 = \kappa_3 = 0$ in the general objective functional provided by Eq.~\eqref{eq:minproblem}, such that the minimisation problem becomes
\begin{equation}
	\label{eq:minproblemsimu}
	 \argmin_{\triplein \in K \cap \Iad} \Jcal (\phi(T)) = \argmin_{\triplein \in K \cap \Iad} \frac{1}{2} \norm{\phi(T) - \phimeas}^2_H,
\end{equation}
and such that $\sigma_0$ and $p_0$ are estimated from $\phi_0$ according to Eq.~\eqref{eq:s0p0}

\subsubsection{Short time horizon}\label{sec:itshortT}
For short time horizons (e.g. $T=10$ days), we will use the Landweber iteration scheme \eqref{eq:landweber_adapt} with the steepest descent step size choice \eqref{adaptive:step} presented in Section~\ref{sec:landweber}. Thus, given a final measurement $\phimeas$, an initial guess $\phi_0^0$ and a maximum number of iterations $j^*$, for any $j=0,\dots,j^*$ we follow the following algorithm:
\begin{enumerate}[font = \bfseries, label = \arabic*.]
    \item Given $\phi_0^j$, set $\sigma_0^j = c_{0,\sigma} + c_{1,\sigma} \phi_0^j$ and $p_0^j = c_{0,p} + c_{1,p} \phi_0^j$.
    \item Run the forward system
    \begin{alignat*}{2}
	& \partial_t \phi - \lambda \Delta \phi + F'(\phi) - m(\sigma) \hh'(\phi) = 0 
	&& \hbox{in $Q_T$,} \\
	& \partial_t \sigma - \eta \Delta \sigma = S_h + S_{ch} \phi - \gamma_h \sigma - \gamma_{ch} \sigma \phi  
	\qquad && \hbox{in $Q_T$,} \\
	& \partial_t p - D \Delta p + \gamma_p p 
	= \alpha_h + \alpha_{ch} \phi 
	&& \hbox{in $Q_T$,} \\
	& \phi = 0, \quad \partial_{\n} \sigma = \partial_{\n} p = 0 
	&& \hbox{on $\Sigma_T$,} \\
	& \phi(0) = \phi_0^j, \quad \sigma(0) = \sigma_0^j, \quad p(0) = p_0^j
	&& \hbox{in $\Omega$,}
    \end{alignat*}
    and call the corresponding solution $(\phi^j, \sigma^j, p^j)$.
    \item Run the adjoint system
    \begin{alignat*}{2}
	& - \partial_t q - \lambda \Delta q + F''(\phi^j) q - m(\sigma^j) \hh''(\phi^j) q + \gamma_{ch} \sigma^j z - S_{ch} z - \alpha_{ch} r = 0 
	\qquad && \hbox{in $Q_T$,} \\
	& - \partial_t z - \eta \Delta z + \gamma_h z + \gamma_{ch} \phi^j z - m'(\sigma^j) \hh'(\phi^j) q = 0 
	&& \hbox{in $Q_T$,} \\
	& - \partial_t r - D \Delta r + \gamma_p r 
	= 0
	&& \hbox{in $Q_T$,} \\
	& q = 0, \quad \partial_{\n} z = \partial_{\n} r = 0 
	&& \hbox{on $\Sigma_T$,} \\
	& q(T) = \phi^j(T) - \phimeas, \quad z(T) = 0, \quad r(T) = 0
	&& \hbox{in $\Omega$,}
    \end{alignat*}
    and call the corresponding solution $(q^j, z^j, p^j)$.
    \item Run the linearised system
    \begin{alignat*}{2}
	& \partial_t Y - \lambda \Delta Y + F''(\phi^j) Y - m(\sigma^j) \hh''(\phi^j) Y - m'(\sigma^j) \hh'(\phi^j) Z = 0 
	\qquad && \hbox{in $Q_T$,} \\
	& \partial_t Z - \eta \Delta Z + \gamma_h Z + \gamma_{ch} \sigma^j Y + \gamma_{ch} \phi^j Z - S_{ch} Y = 0 
	&& \hbox{in $Q_T$,} \\
	& \partial_t P - D \Delta P + \gamma_p P 
	= \alpha_{ch} Y 
	&& \hbox{in $Q_T$,} \\
	& Y = 0, \quad \partial_{\n} Z = \partial_{\n} P = 0 
	&& \hbox{on $\Sigma_T$,} \\
	& Y(0) = q^j(0), \quad Z(0) = z^j(0), \quad P(0) = r^j(0)
	&& \hbox{in $\Omega$.}
    \end{alignat*}
    and call the corresponding solution $(Y^j, Z^j, P^j)$.
    \item Compute the adaptive steepest descent step-size as in \eqref{adaptive:step} with $\kappa_1 = 1$, $\kappa_2 = \kappa_3 = 0$:
    \begin{equation}\label{eq:landweber_step}
    \mu^j = \frac{\norm{q^j(0)}^2_{L^2(\Omega)}}{\norm{Y^j(T)}^2_{L^2(\Omega)}} 
    \end{equation}
    \item Do the Landweber step as in \eqref{eq:landweber_adapt} with $\kappa_1 = 1$, $\kappa_2 = \kappa_3 = 0$:
    \begin{equation}\label{eq:landweber_update}
    \phi_0^{j+1} = \phi_0^j - \mu^j q^j(0).
    \end{equation}
    Check that the updated tumour phase field at $t=0$ verifies $0 \leq \phi_0^{j+1} \leq 1$; otherwise truncate $\phi_0^{j+1}$ accrodingly to ensure that this iterate is admissible.
    \item Check if the following two convergence criteria are satisfied:
    \begin{align*}
        & \mbox{Criterion 1: }\norm{q_0^j}^2_{L^2(\Omega)} \le \eps_{SD} \norm{q_0^0}^2_{L^2(\Omega)}\ \mbox{or}\ \norm{q_0^j - q_0^{j-1}}^2_{L^2(\Omega)} \leq \eps_{SD} \norm{q_0^{j-1}}^2_{L^2(\Omega)},\\
        & \mbox{Criterion 2: } \Jcal (\phi^j(T)) \le \eps_{SD} \Jcal (\phi^0(T))\ \mbox{or}\ \Jcal (\phi^j(T)) -\Jcal (\phi^{j-1}(T))\le \eps_{SD} \Jcal (\phi^{j-1}(T)).
    \end{align*}
    Notice that these two criteria control for convergence of $\phi_0$ and $\phi_T$, respectively. If these convergence criteria are not satisfied, then restart from the top in a new iteration.
\end{enumerate}

\subsubsection{Long time horizon}\label{sec:itlongT}

Unfortunately, the steepest descent step size \eqref{adaptive:step} often becomes very small as the end-time $T$ grows larger, therefore it does not seem the best choice overall, even if we have a good theoretical basis. 
In our case, this can be easily inferred from the results of the simulation study conducted in the next Section \ref{sec:sims}.
Thus, for longer time horizons (e.g. $T=100$ days), we will use the adaptive gradient descent choice proposed in \eqref{eq:malitsky} for the step size $\mu^j$ (see Algorithm 1 in~\cite{MM2019adaptive}). 
We recall that we are still dealing with a Landweber scheme, but with a different choice for the step size. 
However, in the following we will refer to the algorithm below as the adaptive gradient descent method, to differentiate it from the previous one.
Hence, this method requires a final measurement of the tumour phase field $\phimeas$, an initial guess $\phi_0^0$, initial values for the internal parameters $\mu^0$ and $\theta^0$, and a maximum number of iterations $j^*$. Then, for $j=0,\dots,j^*$ we proceed as follows:

\begin{enumerate}[font = \bfseries, label = \arabic*.]
    \item Given $\phi_0^j$, set $\sigma_0^j = c_{0,\sigma} + c_{1,\sigma} \phi_0^j$ and $p_0^j = c_{0,p} + c_{1,p} \phi_0^j$.
    \item Run the forward system
    \begin{alignat*}{2}
	& \partial_t \phi - \lambda \Delta \phi + F'(\phi) - m(\sigma) \hh'(\phi) = 0 
	&& \hbox{in $Q_T$,} \\
	& \partial_t \sigma - \eta \Delta \sigma = S_h + S_{ch} \phi - \gamma_h \sigma - \gamma_{ch} \sigma \phi  
	\qquad && \hbox{in $Q_T$,} \\
	& \partial_t p - D \Delta p + \gamma_p p 
	= \alpha_h + \alpha_{ch} \phi 
	&& \hbox{in $Q_T$,} \\
	& \phi = 0, \quad \partial_{\n} \sigma = \partial_{\n} p = 0 
	&& \hbox{in $\Sigma_T$,} \\
	& \phi(0) = \phi_0^j, \quad \sigma(0) = \sigma_0^j, \quad p(0) = p_0^j
	&& \hbox{in $\Omega$,}
    \end{alignat*}
    and call the corresponding solution $(\phi^j, \sigma^j, p^j)$.
    \item Run the adjoint system
    \begin{alignat*}{2}
	& - \partial_t q - \lambda \Delta q + F''(\phi^j) q - m(\sigma^j) \hh''(\phi^j) q + \gamma_{ch} \sigma^j z - S_{ch} z - \alpha_{ch} r = 0 
	\qquad && \hbox{in $Q_T$,} \\
	& - \partial_t z - \eta \Delta z + \gamma_h z + \gamma_{ch} \phi^j z - m'(\sigma^j) \hh'(\phi^j) q = 0 
	&& \hbox{in $Q_T$,} \\
	& - \partial_t r - D \Delta r + \gamma_p r 
	= 0
	&& \hbox{in $Q_T$,} \\
	& q = 0, \quad \partial_{\n} z = \partial_{\n} r = 0 
	&& \hbox{in $\Sigma_T$,} \\
	& q(T) = \phi^j(T) - \phimeas, \quad z(T) = 0, \quad r(T) = 0
	&& \hbox{in $\Omega$,}
    \end{alignat*}
    and call the corresponding solution $(q^j, z^j, p^j)$.
    \item If $j=0$, set $\mu^0=0.2/q_M$ where $q_M$ is calculated from the global maximum and minimum control variables for $q_0^j$ (see Section~\ref{sec:sdisc}) as 
   \[q_M^j=\max \left\{ \max_{A=1,\dots,n_f}(q_A^j(0)), \left|\min_{A=1,\dots,n_f}(q_A^j(0))\right| \right\}.\]    
    Otherwise, for $j\geq 1$, calculate the next adaptive step size $\mu^j $ as
    \begin{equation}\label{eq:agd_update}
    \mu^j=\min\left\{ \sqrt{1+\theta^{j-1}}\mu^{j-1}, \frac{\norm{\phi_0^j-\phi_0^{j-1}}_{L^2(\Omega)}}{2\norm{q_0^j-q_0^{j-1}}_{L^2(\Omega)}} \right\}.   
    \end{equation}
    \item If $j=0$, set $\theta^0=+\infty$, otherwise update the internal parameter $\theta$ as
     \[ \theta^j = \frac{\mu^j}{\mu^{j-1}}. \]
    \item Calculate the new iterate of the initial tumour phase field as
    \[ \phi_0^{j+1}=\phi_0^j-\mu^jq_0^j. \]
    Check that the updated tumour phase field at $t=0$ verifies $0 \leq \phi_0^{j+1} \leq 1$; otherwise truncate $\phi_0^{j+1}$ accordingly to ensure that this iterate is admissible.
    \item Check if the same  two convergence criteria used for the Landweber iteration scheme in Section~\ref{sec:itshortT} are satisfied. Otherwise, restart from the top in a new iteration.
\end{enumerate}


\section{Simulation study}\label{sec:sims}

In this section, we perform a simulation study to explore the behaviour of the algorithms described in Sections~\ref{sec:itshortT} and \ref{sec:itlongT} to reconstruct the initial tumour phase field considering short and long time horizons respectively. The simulations of the forward prostate cancer model, the linearised problem, and the adjoint problem are carried out using the spatial and temporal discretisation schemes introduced in Sections~\ref{sec:sdisc} and \ref{sec:tdisc}. For simplicity of notation and unless otherwise indicated, in the simulation results provided herein, we drop the superindex $h$  that was introduced in Section~\ref{sec:sdisc} to denote finite-dimensional approximations.

\subsection{Computational scenario and setup}\label{compusetup}
\subsubsection{Model parameters}
We consider an aggressive case of prostate cancer, as detailed in Refs.~\cite{CGLMRR2019,CGLMRR2021}. This scenario could correspond to a tumour exhibiting an intermediate or high Gleason score, which is a fundamental histopathological metric associated with prostate cancer aggressiveness \cite{Mottet2021,Lorenzo2024}.
In the clinic, it is key to early identify these tumours to ensure adequate monitoring and treatment \cite{Mottet2021,Giganti2018,Giganti2021,Lorenzo2024}.
In the context of the reconstruction of the initial conditions of the cancer model considered herein, the aggressive case is also the most interesting: this tumour scenario leads to geometric changes of the tumour over time due to restricted access to the nutrients (see Ref.~\cite{Lorenzo2016,CGLMRR2019}), which pose a challenge for the estimation of the initial tumour phase-field map. We use the same parameters of the model as in Ref.~\cite{CGLMRR2021} for the simulations presented in this work, which are based on previous studies in the literature \cite{Berges1995,CGLMRR2019,Lorenzo2016,Lorenzo2017,Lorenzo2019,Lorenzo2024,Mottet2021,Schmid1993,Xu2016,Giganti2018,Giganti2021}.

\subsubsection{Numerical implementation details}
The numerical algorithms described in Section~\ref{sec:numeth} to solve the forward, linearised, and adjoint problems as well as to reconstruct the initial tumour phase field were implemented using our in-house isogeometric codes to simulate prostate cancer growth \cite{Lorenzo2016,Lorenzo2017,Lorenzo2019,CGLMRR2019,CGLMRR2021}. These codes were developed following the general directions and algorithms provided in Ref.~\cite{Cottrell2009}.

The numerical simulations for the reconstruction algorithms considered in this work are performed in a square tissue patch with side length $L_d=$\SI{3000}{\micro\meter}. This computational domain is discretised using 256 $C^1$-continuous quadratic B-spline elements per side. We further consider a constant time step, which is set to $\Delta t_n=0.1$ days for the forward and linearised problems. As the adjoint problem runs backwards in time, we set $\Delta t_n=-0.1$ days for its time discretisation.

We set the tolerance of the Newton-Rapshon method to $\eps_{NR}=10^{-3}$. The convergence of the GMRES method is set to a tolerance of $\eps_{GMRES}=10^{-3}$ and a maximum of 500 iterations. The convergence of the Landweber iteration scheme and the adaptive gradient descent algorithm used to reconstruct the initial tumour phase field is fixed at a tolerance of $\eps_{SD}=10^{-4}$ and a maximum of 500 iterations for both methods.

\subsubsection{Ground truth}
To assess the performance of the tumour phase field reconstruction methods described in Sections~\ref{sec:itshortT} and \ref{sec:itlongT}, we consider a reference simulation of the prostate cancer model.
This \emph{in silico} ground truth is generated with a finer mesh within the same spline space defined above.
In particular, we use twice the number of elements along each spatial direction (i.e., $512\times 512$ elements).
We consider an initial phase field configuration consisting of an ellipsoidal tumour placed in the centre of the computational domain (i.e., $(x_c,y_c)=(L_d/2, L_d/2)$) and with semi-axes $a=$\SI{150}{\micro\meter} and $b=$\SI{200}{\micro\meter} parallel to the sides of the domain. We implement this initial condition \emph{via} the $L^2$-projection of the hyperbolic tangent function
\begin{equation}\label{eq:iniphi}
\phi_0(x_1,x_2) = \frac{1}{2} - \frac{1}{2}\tanh\left( 10\left(\sqrt{\frac{(x_1-x_c/2)^2}{a^2} + \frac{(x_2-y_c/2)^2}{b^2}} - 1\right) \right)
\end{equation}
over the $C^1$-continuous quadratic B-spline space supporting our spatial discretisation. 
This operation provides the control variables $\phi_{0,A}=\phi_{A}(0)$, $A=1,\ldots,n_f$, for the spline representation of the phase-field initial condition, i.e., $\phi^h_0(x)=\phi^h(0,x)=\sum_{A=1}^{n_f}\phi_{0,A}N_A(x)$ (see Section~\ref{sec:sdisc}).
In particular, this ground truth simulation also provides the terminal measurement $\phimeas$.
Additionally, we analyse the performance of the reconstruction algorithms in the presence of noise. 
Towards this end, we affected the tumour measurement at the time horizon $\phimeas$ with 10\% Gaussian noise, which was applied to the control variables $\phi_{\text{meas},A}>0.001$.
This implementation follows the usual thresholding segmentation approaches used in the construction of tumour-defining spatial maps from imaging data for their ensuing use in biophysical modelling \cite{Lorenzo2024,agosti2018personalized}.

\subsubsection{Initial guess}\label{sec:iniguess}
To start the reconstruction algorithms, the initial guess $\phi_0^0$ consists of a small spherical tumour implemented using Eq.~\eqref{eq:iniphi} with radius $a=b=$\SI{100}{\micro\meter} and centre coordinates $(x_c,y_c)$ matching the centre of mass of the tumour phase field map $\phimeas$ measured at $t=T$, which is used to define the objective functional of the minimisation problem (see Eq.~\eqref{eq:minproblemsimu}). The rationale for this choice is having a sufficiently small tumour aligned with the observed tumour configuration at $t=T$ as a reasonable starting point to facilitate the iterative reconstruction of $\phi_0^0$. 

\subsubsection{Metrics to assess the reconstruction of the tumour phase field}
We assess the tumour phase field at $t=0$ and $t=T$ in terms of a panel of four metrics that are commonly used in the literature of mathematical models of cancer \cite{Lorenzo2022_review,Lorenzo2024, Hormuth2021,Jarrett2018,Wong2016}. First, we use the relative error in tumour volume $V_\phi$. The latter is calculated \emph{via} spatial integration of the tumour phase field over the volume enclosed by the isosurface $\phi=0.5$, which we denote by $\Omega_{\phi}$. This isosurface implicitly tracks the interface between healthy and tumour tissue \cite{Lorenzo2016,Lorenzo2019,CGLMRR2019,CGLMRR2021}. Thus, we calculate $V_\phi$ as 
\begin{equation*}
    V_\phi = \int_{\Omega_{\phi}} \, \de x, 
\end{equation*}
and the relative error in tumour volume as
\begin{equation*}
    e_{V,t}=e_V (t)= \frac{V_\phi^{ref}(t)-V_\phi^{rec}(t)}{V_\phi^{ref}(t)},
\end{equation*}
where $V_\phi^{ref}$ is the volume of the tumour calculated from the tumour phase field of the reference (i.e., ground truth) simulation $\phi^{ref}$ and $V_\phi^{rec}$ is the volume of the tumour calculated from the reconstructed tumour phase field $\phi^{rec}$.
Second, we use the S{\o}rensen-Dice similarity coefficient (DSC) to assess whether the volume of the reconstructed tumour phase field matches the one from the reference configuration at the same time \cite{Hormuth2021,Wong2016,Lorenzo2024}. Towards this end, we calculate the DSC at time $t$ as
\begin{equation*}
    \mathrm{DSC}_t=\mathrm{DSC}(t)=\frac{2V_\phi^{int}(t)}{V_\phi^{rec}(t)+V_\phi^{ref}(t)},
\end{equation*}
where $V_\phi^{int}$ is the volume of the intersection of the tumour phase fields from the ground truth and resulting from the reconstruction procedure. The volumes to calculate the DSC are also calculated with respect to the corresponding isosurfaces $\phi=0.5$ for the reference and reconstructed tumour phase fields. Third, we use the relative $L^2$ error of the reconstructed tumour phase field with respect to its counterpart in the reference simulation. Hence, the $L^2$ error is calculated as
\begin{equation*}
    e_{L^2,t}=e_{L^2}(t) = \frac{ \norm{\phi^{ref}(t) - \phi^{rec}(t)}_{\Lx2} }{ \norm{ \phi^{ref}(t) }_{\Lx2} }
\end{equation*}
Finally, we also calculate the concordance correlation coefficient (CCC) \cite{Lin1989,Hormuth2021,Jarrett2018,Lorenzo2024} to assess the pointwise agreement between the reference and reconstructed tumour phase fields as
\begin{equation*}
    \mbox{CCC}_t= \mbox{CCC}(t)=\frac{2 \mbox{Cov}(\phi^{ref}(t),\phi^{rec}(t))}{\mbox{Var}(\phi^{ref}(t)) +\mbox{Var}(\phi^{rec}(t)) + (\mbox{Mean}(\phi^{ref}(t))-\mbox{Mean}(\phi^{rec}(t)))^2},
\end{equation*}
where the covariance, variances, and means of the reference and reconstructed tumour phase fields are calculated over the volume enclosed by the corresponding $\phi=0.5$ isosurfaces.

\subsection{Reconstruction of the tumour phase field for short time horizons}\label{res:shortt}

\begin{figure}[!t]   
    \includegraphics[width=\textwidth]{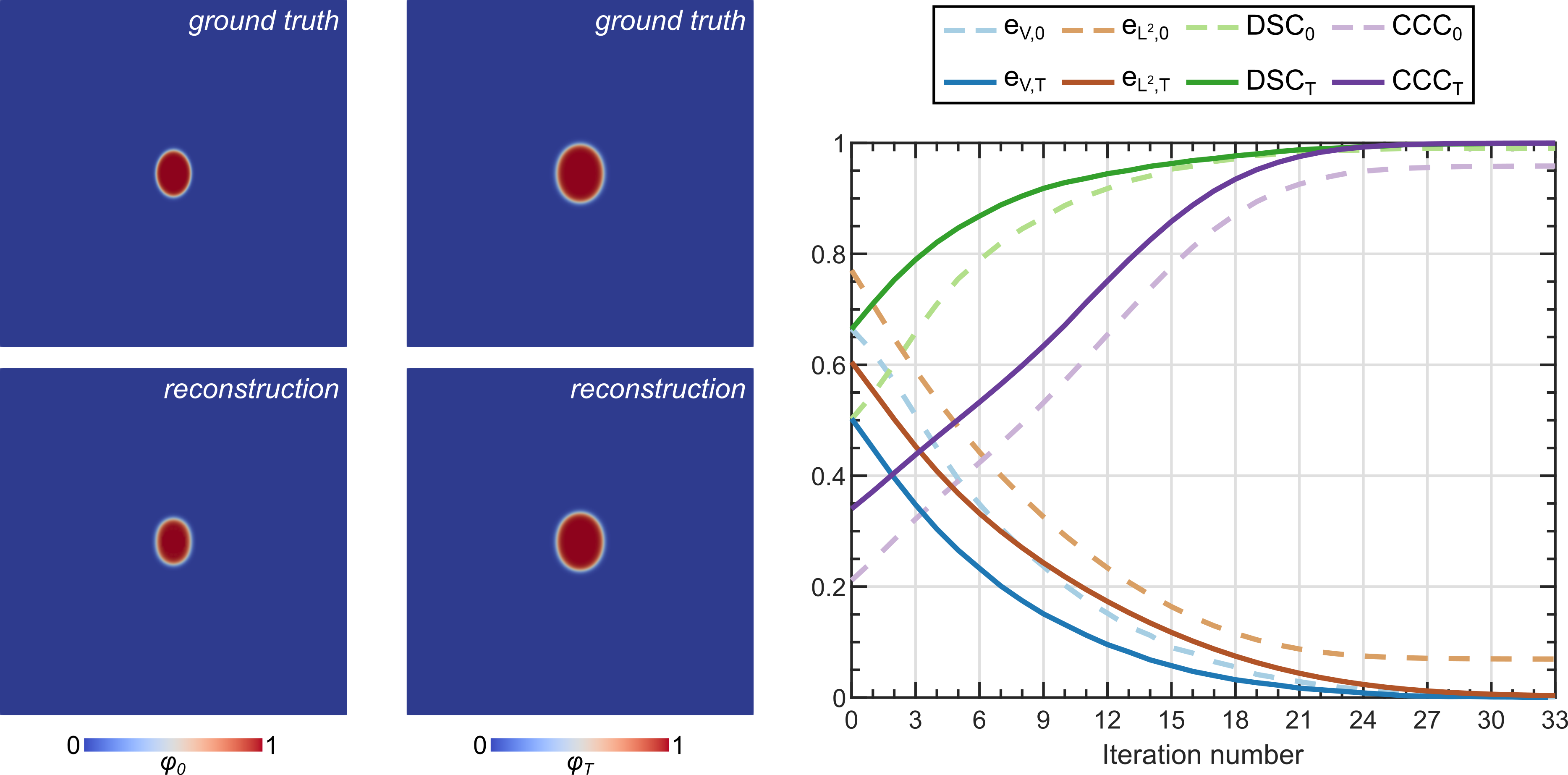} 
    \caption{Reconstruction of the initial tumour phase field from a measurement at $T=15$ days using the Landweber iteration scheme. The first two columns compare the tumour phase field from the reference simulation and the corresponding reconstruction at $t=0$ and $t=T$. The plot in the last column provides the values of the four metrics used to assess the reconstruction of the tumour phase field in each iteration of the Landweber algorithm. }
    \label{fig:landwebert15_metrics}
  \end{figure}

  \begin{figure}[!t]   
    \includegraphics[width=\textwidth]{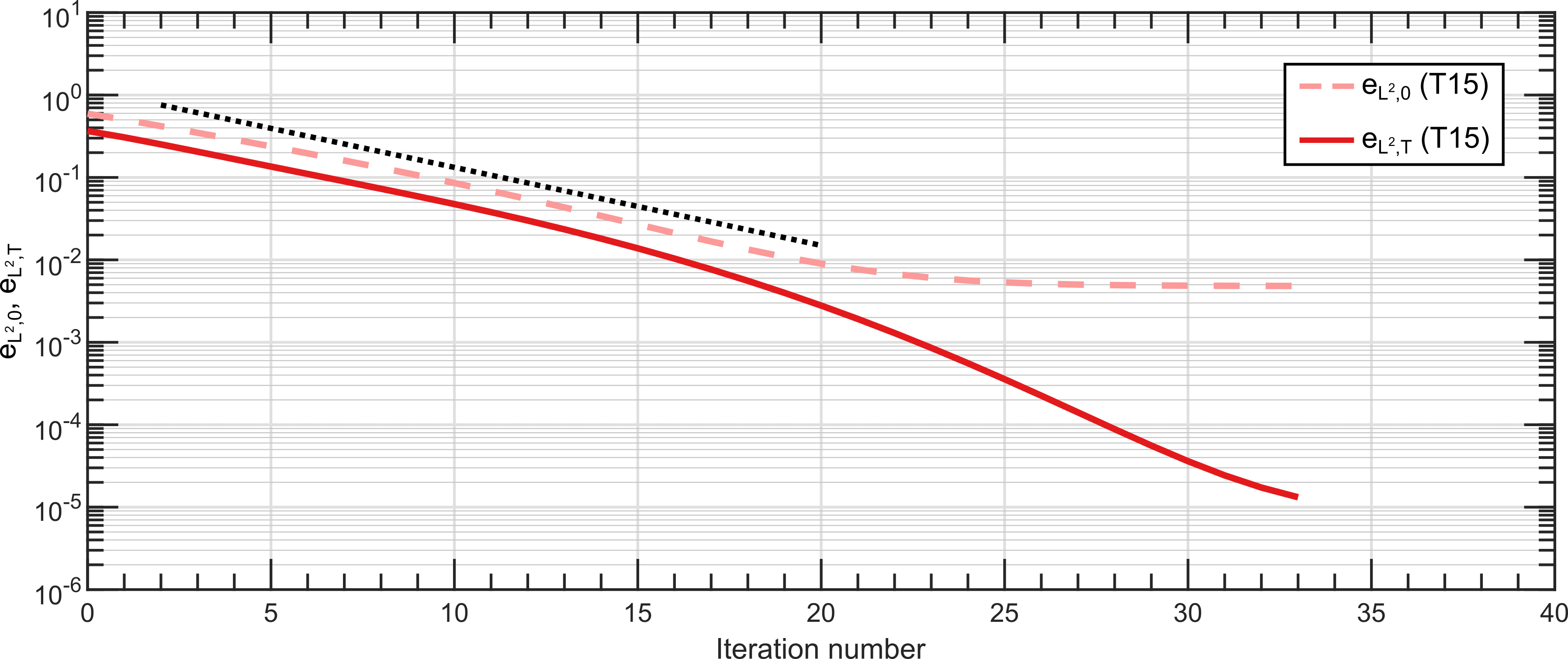} 
    \caption{Convergence of the Landweber algorithm for $T= 15$ days. This figure provides the changes of the relative $L^2$ error of the reconstructed tumour phase field at $t=0$ and $t=T$. For $e_{L^2,0}$, we also plot a straight dotted line with the same slope as the linear portion of the trajectory of the relative $L^2$ error. The value of this slope is -9.47$\cdot 10^{-2}$, which corresponds to a value of parameter $c$ of 1.96$\cdot 10^{-1}$ in Theorem~\ref{thm:landweber}. Thus, the convergence in all scenarios is infralinear, as it was found theoretically.}
    \label{fig:landweber_conv}
  \end{figure}

We analyse the performance of the Landweber iteration scheme in a computational scenario with a short time horizon of $T=15$ days.
Figure~\ref{fig:landwebert15_metrics} provides the reconstruction of the tumour phase field at $t=0$ and at the time horizon $t=T$, as well as a comparison of these results to the ground truth simulation at the same time points. 
Qualitatively, we observe that the Landweber method yields a very good reconstruction of the tumour phase field at $t=0$, and that this results in an excellent match to the measurement of the tumour phase field from the ground truth at $t=T$, which was used to drive the inverse problem.
Figure ~\ref{fig:landwebert15_metrics} also provides the values of the metrics used to assess the reconstruction of the tumour phase field at both $t=0$ and $t=T$ in each iteration of the Landweber scheme until reaching convergence.
These quantitative results demonstrate that this method achieves an accurate reconstruction of the tumour phase field at both $t=0$ and $t=T$. 
Indeed, these quantitative results show a minimal error at $t=T$, thereby confirming that the reconstruction algorithm yields a perfect match to the available measurement of the tumour phase field at the time horizon.
Furthermore,  we also obtain a very small error in the tumour volume and a high DSC for the reconstructed tumour phase field at $t=0$, which indicate an excellent reconstruction of the geometry of the initial tumour.
While the CCC and the $L^2$ error at $t=0$ reveal some mismatch between the reconstructed tumour and the ground truth at $t=0$, these results still represent a successful recovery of the tumour phase field as compared with the corresponding values of these metrics obtained in the assessment of model calibration and forecasting in the literature \cite{Wu2022, Hormuth2021, Lorenzo2024, Wong2016}.
Moreover, the trajectory of the values of all assessment metrics during the iterations plateaus towards the end of the reconstruction procedure, which suggests the achievement of sufficient convergence for the chosen tolerance.
In Figure~\ref{fig:landweber_conv} we provide more detail of the convergence rate of the Landweber iteration scheme according to the relative $L^2$ error.
In particular, the slope of the linear part of the trajectory of the relative $L^2$ error of the reconstructed initial tumour phase field along the Landweber iterations is -9.47$\cdot 10^{-2}$. 
This slope also corresponds to a value of parameter $c$ in Theorem~\ref{thm:landweber} of  1.96$\cdot 10^{-1}$.
Thus, Figure~\ref{fig:landweber_conv}  shows that the convergence rate of the Landweber reconstruction algorithm is infralinear in accordance with Theorem~\ref{thm:landweber}.

\begin{figure}[!t]
    \begin{subfigure}[t]{0.2\textwidth}
      \includegraphics[width=\textwidth]{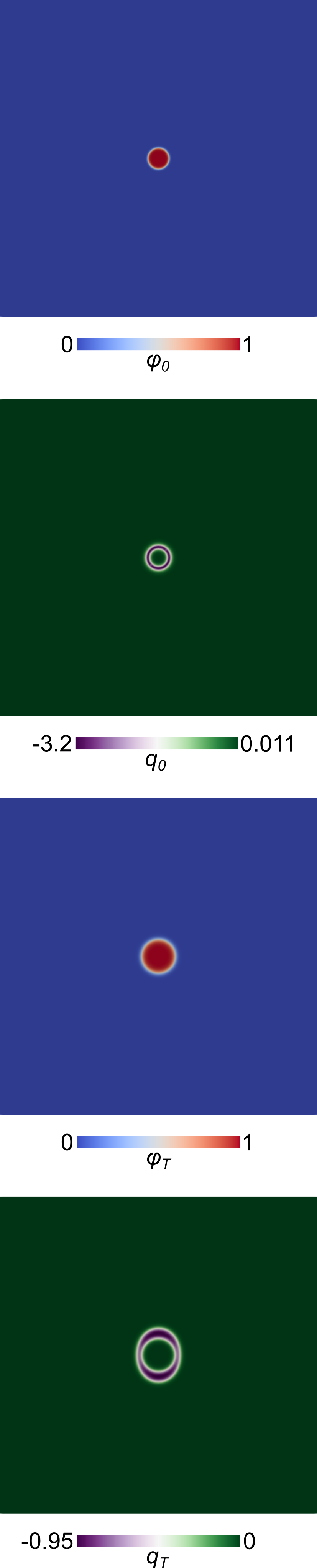}
      \caption{$j=0$}
    \end{subfigure}
    \hfill
    \begin{subfigure}[t]{0.2\textwidth}
      \includegraphics[width=\textwidth]{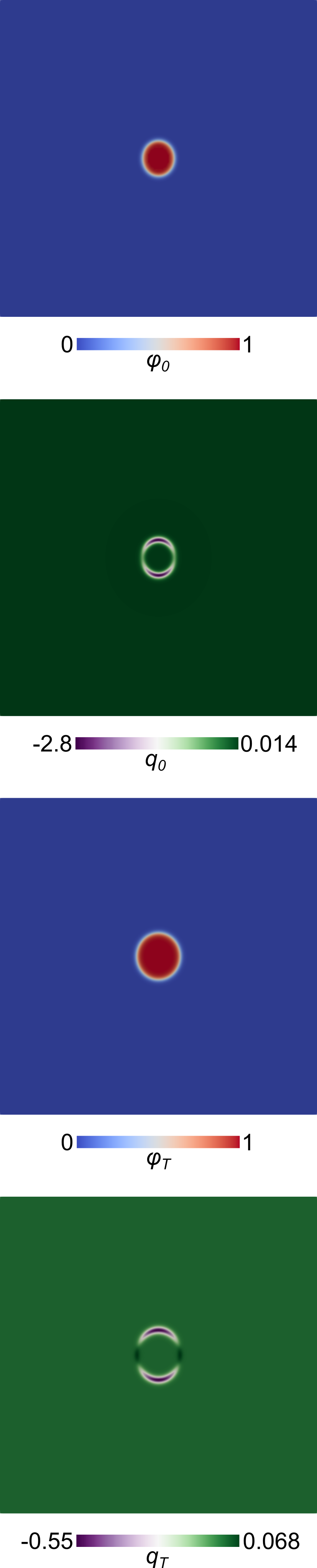}
      \caption{$j=11$}
    \end{subfigure}
    \hfill
    \begin{subfigure}[t]{0.2\textwidth}
      \includegraphics[width=\textwidth]{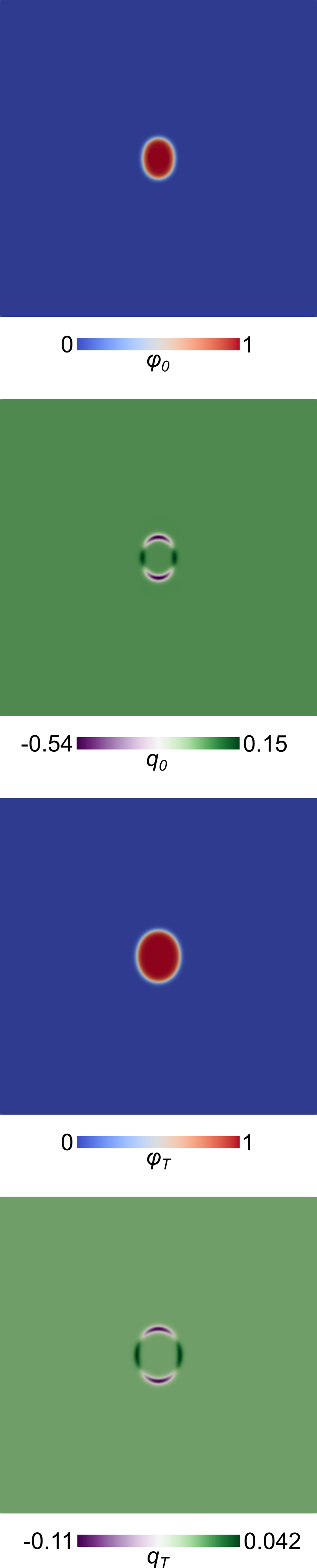}
      \caption{$j=22$}
    \end{subfigure}
    \hfill
    \begin{subfigure}[t]{0.2\textwidth}
      \includegraphics[width=\textwidth]{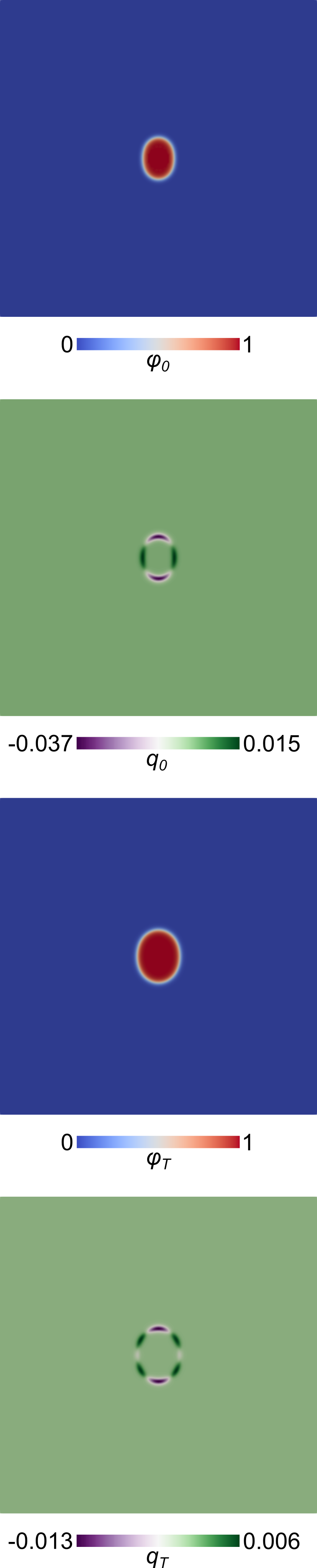}
      \caption{$j=33$}
    \end{subfigure}
    \hfill
    \caption{Reconstruction of the initial tumour phase field from a measurement at $T=15$ days using the Landweber iteration scheme. In each panel, the first two rows represent the tumour phase field and its corresponding adjoint variable at $t=0$ (i.e, $\phi_0$ and $q_0$), while the last two rows provide the same quantities at $t=T$ (i.e, $\phi_T$ and $q_T$). }
    \label{fig:landwebert15_iters}
  \end{figure}

To further understand the reconstruction procedure with the Landweber iteration scheme, Figure~\ref{fig:landwebert15_iters} shows the tumour phase field and its associated dual variable ($q$, which is used to update the iterates of the initial tumour phase field) at $t=0$ and $t=T$ during several iterations of the Landweber algorithm.
The results depicted in Figure~\ref{fig:landwebert15_iters} show how this reconstruction method progressively updates the starting guess of the initial tumour phase field to optimally match the measurement of the tumour at the time horizon.
We observe that the initial guess of the tumour is first increased in size in all directions of space.
Once the iterates of $\phi_0$ approach the size of the smallest axis of the ellipsoidal tumour in the ground truth at $t=0$, the algorithm then continues enlarging $\phi_0$ in the direction of the longest axis. 
Then, as the reconstruction procedure approaches convergence, the updates of the initial tumour phase field are increasingly focused along the interface between tumour and healthy tissue.
We note that these updates are obtained by solving the dual problem, which is initialised using the mismatch between the tumour phase field and its corresponding measurement at time $t=T$.
Hence, the spatial map of $q_T$ in the computational scenarios considered herein initially shows values between -1 and 0, in which the negative values indicate regions that need to be occupied by the tumour.
Towards this end, solving the dual problem maps this mismatch at time $t=T$ into precise updates of the tumour phase field at $t=0$.
As the reconstruction method proceeds, the mismatch between the tumour measurement and model reconstruction at the time horizon progressively diminishes, which progressively reduces the range of values of $q_T$ towards zero.
This results in a parallel contraction of the value range and increased localisation of  $q_0$ to update the initial tumour phase field.
Additionally, the Landweber adaptive step size exhibits a median (range) of 3.53$\cdot 10^{-2}$ (2.70$\cdot 10^{-2}$, 9.09$\cdot 10^{-2}$).
Hence, the product of these step size values by the increasingly lower values of $q_0$ during the iterations of the Landweber algorithm results in a progressively smaller update to the initial tumour phase field, which further translates into minimal changes to $\phi_T$ as the algorithm approaches convergence.

\subsection{Reconstruction of the tumour phase field for long time horizons}\label{res:longt}

\begin{figure}[!t]   
    \includegraphics[width=\textwidth]{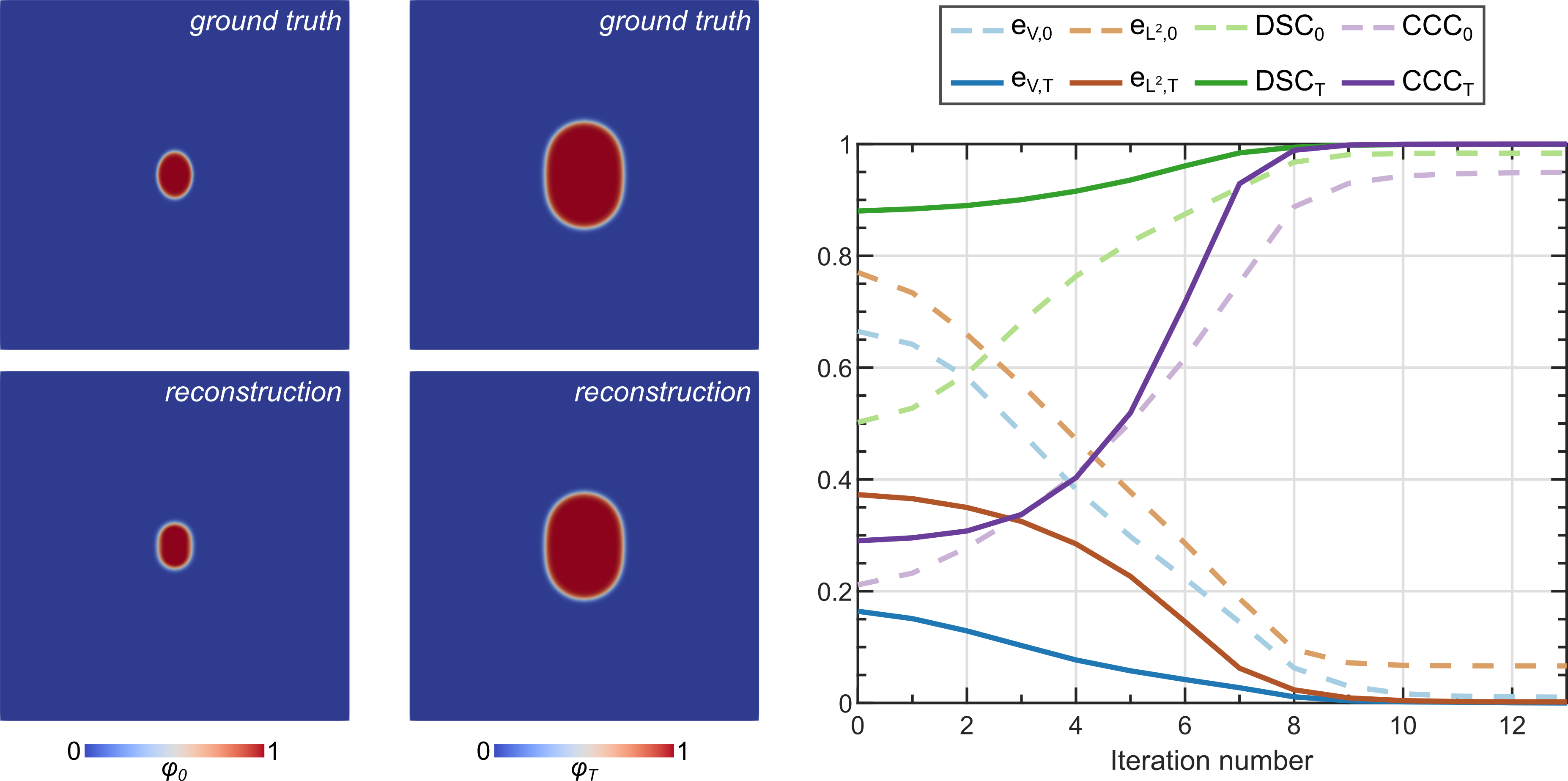} 
    \caption{Reconstruction of the initial tumour phase field from a measurement at $T=90$ days using the adaptive gradient descent algorithm. The first two columns compare the tumour phase field from the reference simulation and the corresponding reconstruction at $t=0$ and $t=T$. The plot in the last column provides the values of the four metrics used to assess the reconstruction of the tumour phase field in each iteration of the adaptive gradient descent algorithm.}
    \label{fig:agd90_metrics}
  \end{figure}

\begin{figure}[!t]   
    \includegraphics[width=\textwidth]{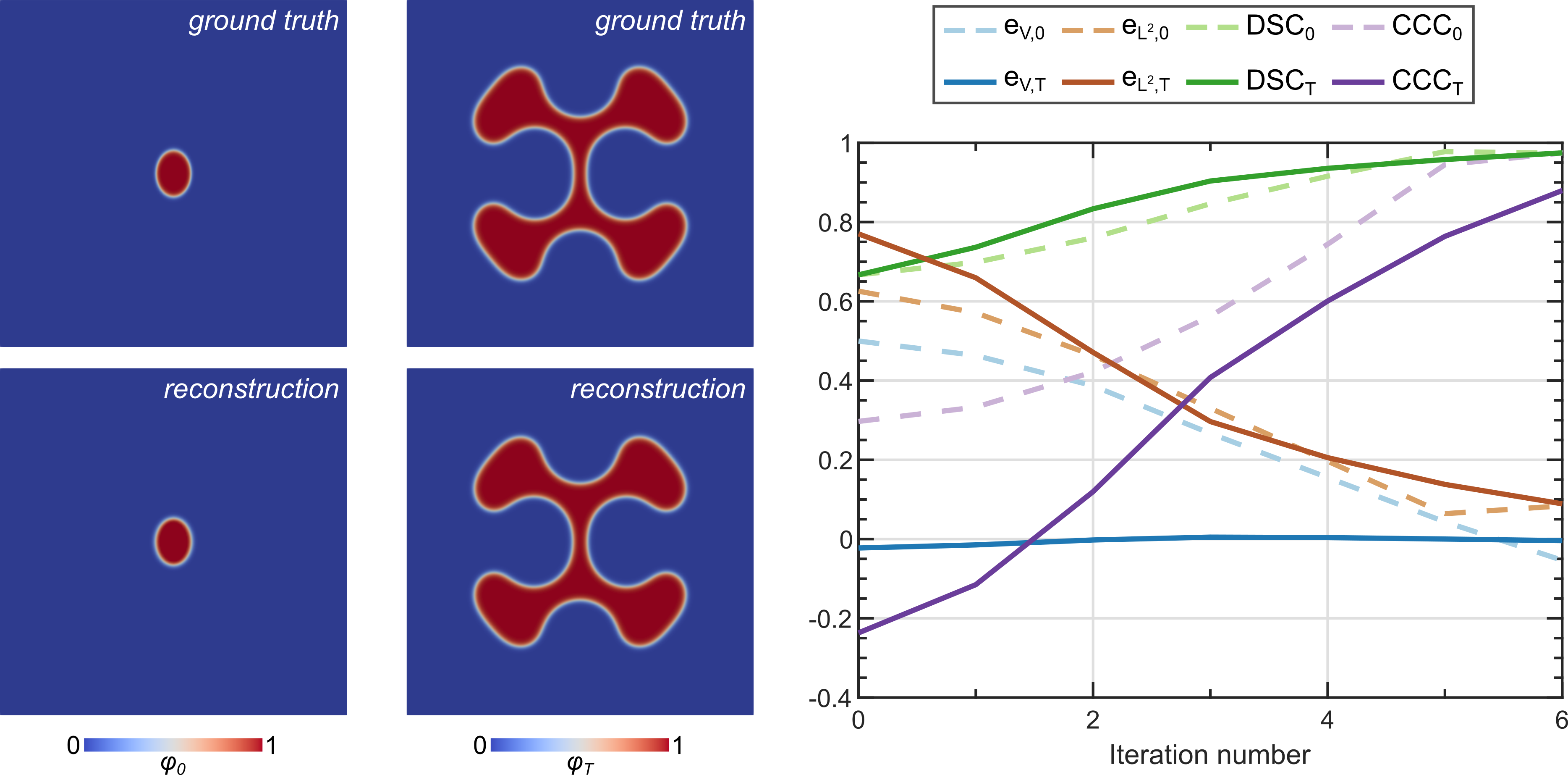} 
    \caption{Reconstruction of the initial tumour phase field from a measurement at $T=365$ days using the adaptive gradient descent algorithm. The first two columns compare the tumour phase field from the reference simulation and the corresponding reconstruction at $t=0$ and $t=T$. The plot in the last column provides the values of the four metrics used to assess the reconstruction of the tumour phase field in each iteration of the adaptive gradient descent algorithm.}
    \label{fig:agd365_metrics}
  \end{figure}

    \begin{figure}[!t]
    \begin{subfigure}[t]{0.2\textwidth}
      \includegraphics[width=\textwidth]{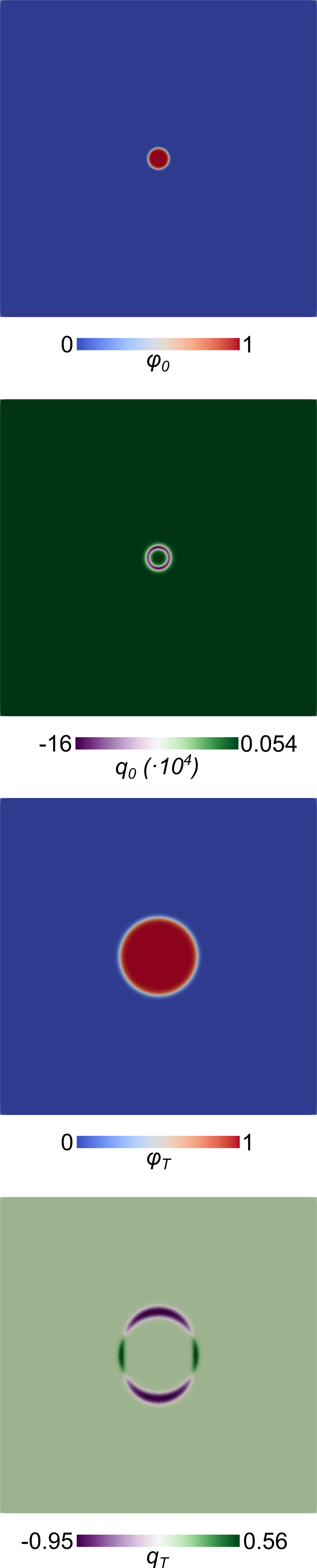}
      \caption{$j=0$}
    \end{subfigure}
    \hfill
    \begin{subfigure}[t]{0.2\textwidth}
      \includegraphics[width=\textwidth]{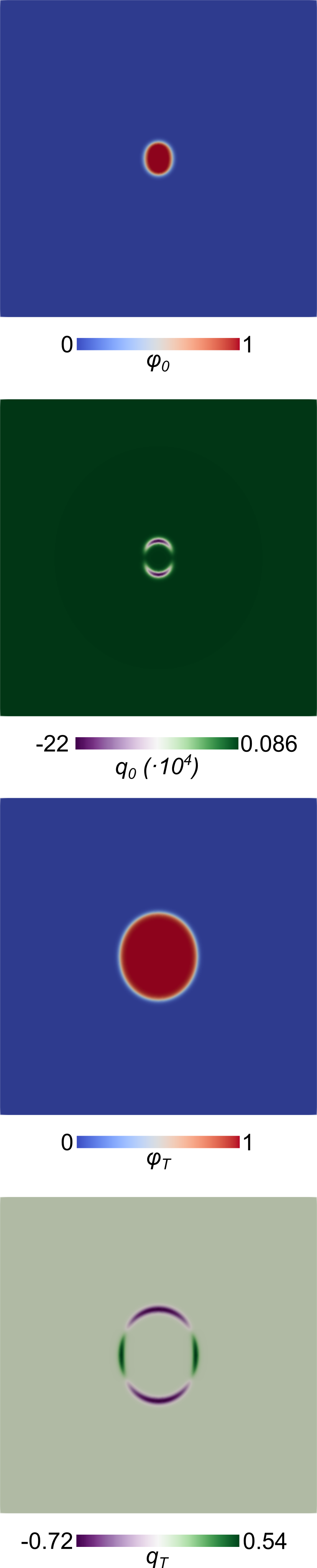}
      \caption{$j=5$}
    \end{subfigure}
    \hfill
    \begin{subfigure}[t]{0.2\textwidth}
      \includegraphics[width=\textwidth]{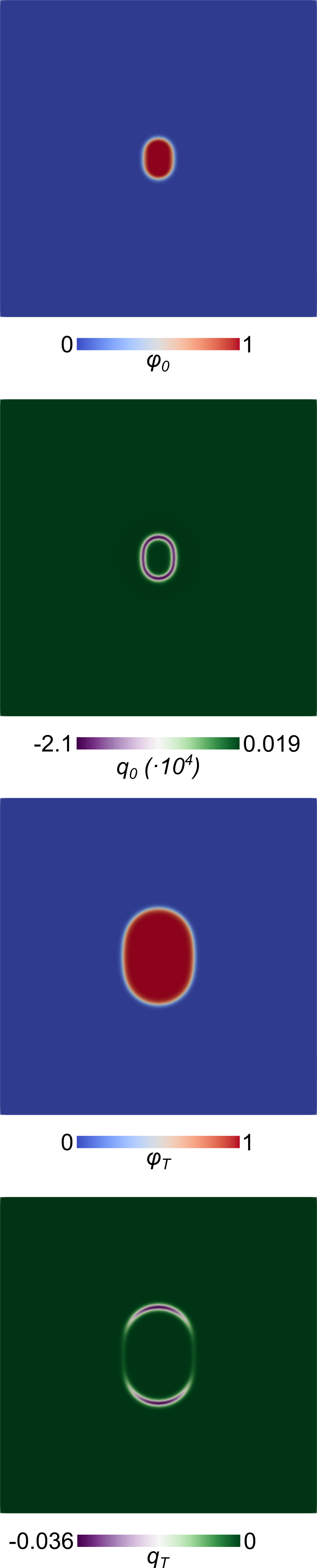}
      \caption{$j=9$}
    \end{subfigure}
    \hfill
    \begin{subfigure}[t]{0.2\textwidth}
      \includegraphics[width=\textwidth]{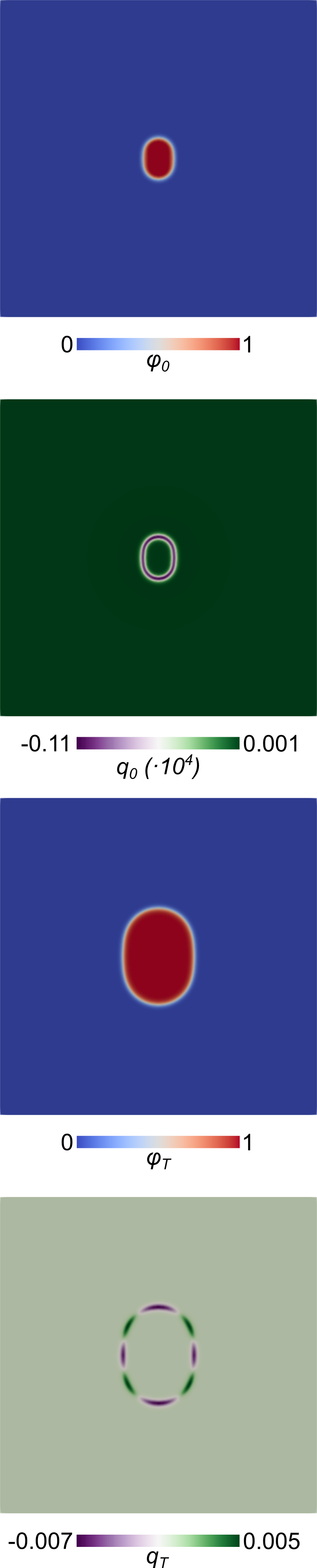}
      \caption{$j=13$}
    \end{subfigure}
    \hfill
    \caption{Reconstruction of the initial tumour phase field from a measurement at $T=90$ days using the adaptive gradient descent algorithm. In each panel, the first two rows represent the tumour phase field and its corresponding adjoint variable at $t=0$ (i.e, $\phi_0$ and $q_0$), while the last two rows provide the same quantities at $t=T$ (i.e, $\phi_T$ and $q_T$). }
    \label{fig:agd90_iters}
  \end{figure}

    \begin{figure}[!t]
    \begin{subfigure}[t]{0.2\textwidth}
      \includegraphics[width=\textwidth]{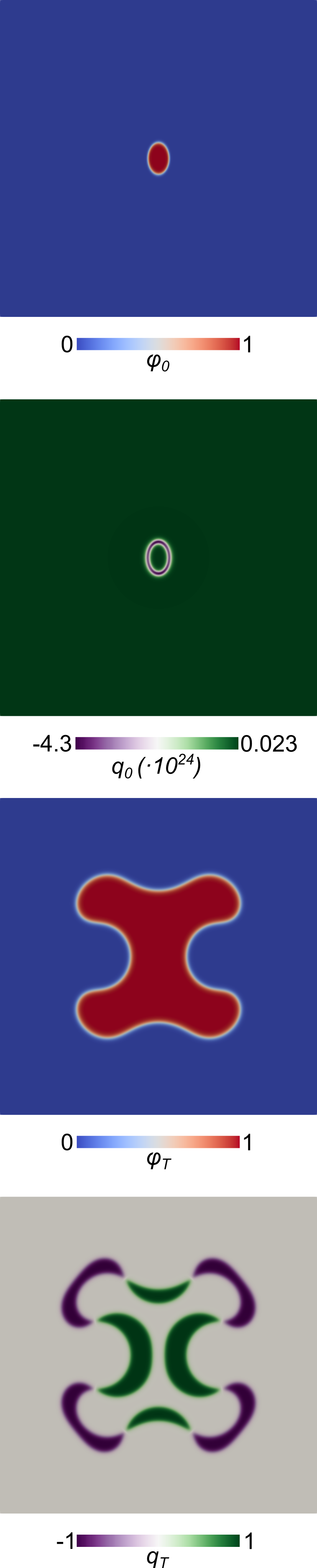}
      \caption{$j=0$}
    \end{subfigure}
    \hfill
    \begin{subfigure}[t]{0.2\textwidth}
      \includegraphics[width=\textwidth]{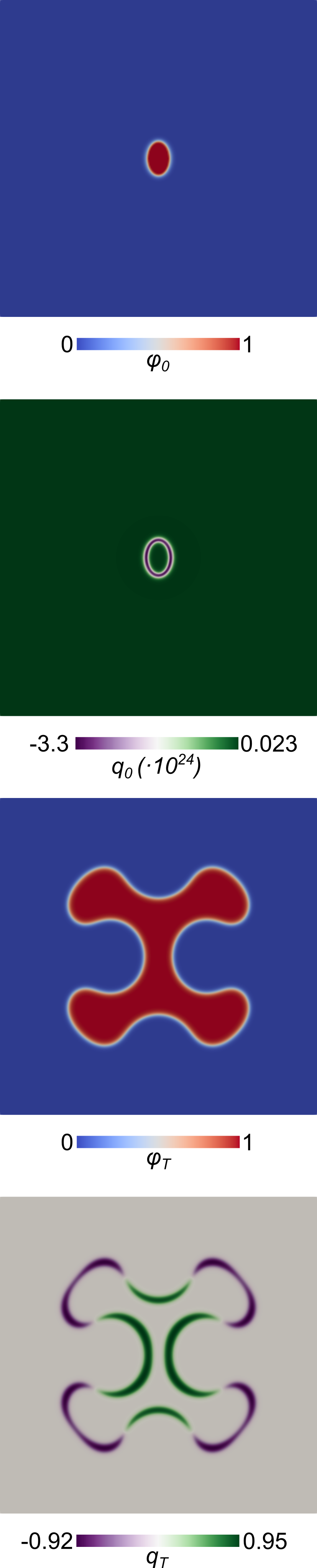}
      \caption{$j=2$}
    \end{subfigure}
    \hfill
    \begin{subfigure}[t]{0.2\textwidth}
      \includegraphics[width=\textwidth]{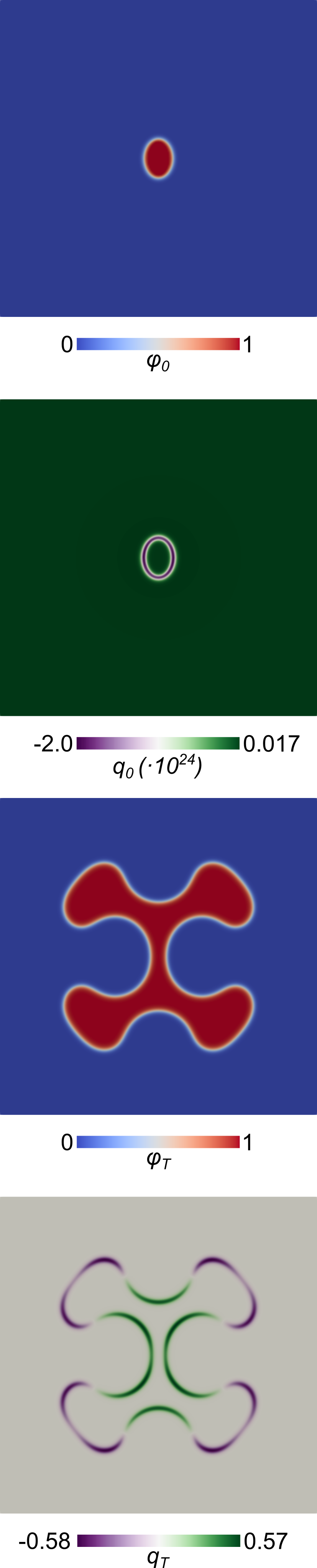}
      \caption{$j=4$}
    \end{subfigure}
    \hfill
    \begin{subfigure}[t]{0.2\textwidth}
      \includegraphics[width=\textwidth]{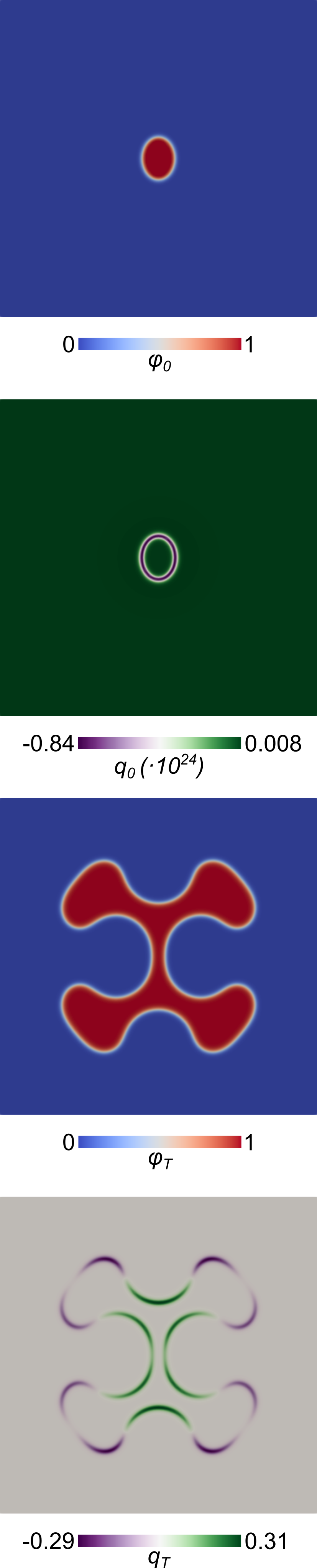}
      \caption{$j=6$}
    \end{subfigure}
    \hfill
    \caption{Reconstruction of the initial tumour phase field from a measurement at $T=365$ days using the adaptive gradient descent algorithm. In each panel, the first two rows represent the tumour phase field and its corresponding adjoint variable at $t=0$ (i.e, $\phi_0$ and $q_0$), while the last two rows provide the same quantities at $t=T$ (i.e, $\phi_T$ and $q_T$).}
    \label{fig:agd365_iters}
  \end{figure}

We study the performance of the adaptive gradient descent algorithm in two computational scenarios with a long time horizon of $T=90$ and $365$ days.
Figure~\ref{fig:agd90_metrics} shows that this method can reconstruct the earlier state of the tumour considering a time horizon of $T=90$ days efficiently and accurately.
Additionally, the results in Figure~\ref{fig:agd365_metrics} confirm that the adaptive gradient descent method can reconstruct the tumour state a year before acquiring the spatial measurement (i.e., $T=365$ days).
Nevertheless, for this last scenario, we increased the tolerance $\varepsilon_{SD}$ to $0.1$ and we had to change the initial guess from a small round tumour (as indicated in Section~\ref{sec:iniguess}) to an ellipsoidal tumour with larger vertical axis ($a=$\SI{100}{\micro\meter}, $b=$\SI{150}{\micro\meter}).
The tolerance increment is implemented based on the results of a first simulation with the original value indicated in Section~\ref{compusetup}, which showed that it was too strict for this case.
This issue could be motivated by the drastic geometric changes of the tumour phase field at $t=T$, which may accumulate small local reconstruction errors over a large tumour volume.
The rationale for adapting the initial guess is that the round tumour used in previous simulations led to a failure of the reconstruction algorithm, whereby the size of $\phi_0$ progressively decreases in each iteration of the adaptive gradient descent method until vanishing.
We believe that this issue stems from a choice of an initial guess that is too far from the ground truth in this computational scenario.
Thus, we opted to change the initial guess according to the known dynamics of the model.
In particular, the simulations of the model presented in \cite{CGLMRR2019} show that the horizontal tumour branches observed at $t=T$ come from an ellipsoidal tumour with a longer axis in the orthogonal direction to the branches at time $t=0$, as observed in the simulation used to generate the ground truth in this work.
After adjusting the initial guess to match this knowledge from the model dynamics, we observe an excellent performance of the adaptive gradient descent.
We would like to remark that the logic leveraged to adjust the initial guess can also be followed in actual clinical scenarios: given a spatial tumour measurement and a known model, one can choose the initial guess of the early state of the tumour phase field that favours a model prediction of the tumour morphology observed in the measurement.

Figures~\ref{fig:agd90_metrics} and \ref{fig:agd365_metrics} demonstrate that the adaptive gradient descent method achieves a high qualitative and quantitative agreement between the reconstructed tumour phase field and the ground truth both at $t=0$ and $t=T$ in the two computational scenarios considered in this section.
As observed for the Landweber method, the reconstruction yields very good results at time $t=T$ according to all metrics.
While the DSC and relative error in the tumour volume at $t=0$ show that the reconstructed tumour at $t=0$ is practically identical to the ground truth, the CCC and the relative $L^2$ error suggest that there is a certain mismatch.
Nevertheless, the values of these metrics at time $t=0$ still represent a successful reconstruction of the initial tumour morphology, and they are also comparable to prediction errors obtained in the computational tumour forecasting studies in the literature \cite{Hormuth2021, Wong2016, Wu2022, Lorenzo2024}.
For the $T=90$ day scenario, the plateauing trend in all metrics at $t=0$ and $t=T$ suggests that the adaptive gradient descent method has already reached convergence for the tolerance considered in the simulation.
Thus, these results further support the use of a larger tolerance to obtain equally acceptable reconstruction results in these scenarios and, hence, facilitate convergence in more demanding cases.
Indeed, this is the case of the $T=365$ day scenario, in which we observe a progressive improvement of the metrics over the iterations of the adaptive gradient descent algorithm to satisfactory values according to computational oncology literature \cite{Hormuth2021, Wong2016, Wu2022, Lorenzo2024}.
In this computational scenario, only DSC$_0$, DSC$_T$, CCC$_0$, and e$_{V,T}$ start to reach a plateau towards the end of the simulation.
The lack of a generalised plateauing trend results from the earlier termination of the adaptive gradient descent algorithm caused by the increase in convergence tolerance $\varepsilon_{SD}$. 

Figures~\ref{fig:agd90_iters} and \ref{fig:agd365_iters} further show the changes in the tumour phase field $\phi$ and the dual variable $q$ at times $t=0$ and $t=T$ obtained in four iterations of the reconstruction procedure in the two computational scenarios in this section.
The results for $T=90$ days are similar to those obtained with the Landweber method: the updates of the initial tumour phase field $\phi_0$ obtained from $q_0$ initially enlarge the tumour in both directions of space, then modify the tumour shape in the direction with larger mismatch with respect to the ground truth (i.e., in the vertical direction), and finally adjust the morphology of $\phi_0$ along the interface.
In the $T=365$ day scenario, we observe that the initial ellipsoidal tumour progressively adjusts its shape while increasing its size until the algorithm reaches convergence.
Additionally, for both computational scenarios in this section, we observe that the absolute values and range of the dual variable $q$ at both $t=0$ and $t=T$ decrease globally as the reconstruction advances and the mismatch between the tumour reconstruction and the ground truth diminishes, as we had also observed with the Landweber method.
Figures~\ref{fig:agd90_iters} and \ref{fig:agd365_iters}  also show that the values of $q_0$ increase as we consider a higher time horizon $T$.
To ensure a stable update of the initial tumour phase field, the adaptive gradient descent algorithm produces a median (range) value of the adaptive step size of 1.51$\cdot 10^{-6}$ (9.76$\cdot 10^{-7}$, 1.96$\cdot 10^{-6}$) and 1.02$\cdot 10^{-25}$ (4.24$\cdot 10^{-26}$, 1.58$\cdot 10^{-25}$) for $T=90$ and 365 days, respectively.

\subsection{Reconstruction of the tumour phase field using ground truth affected by noise}\label{res:noise}

\begin{figure}[!t]   
    \includegraphics[width=\textwidth]{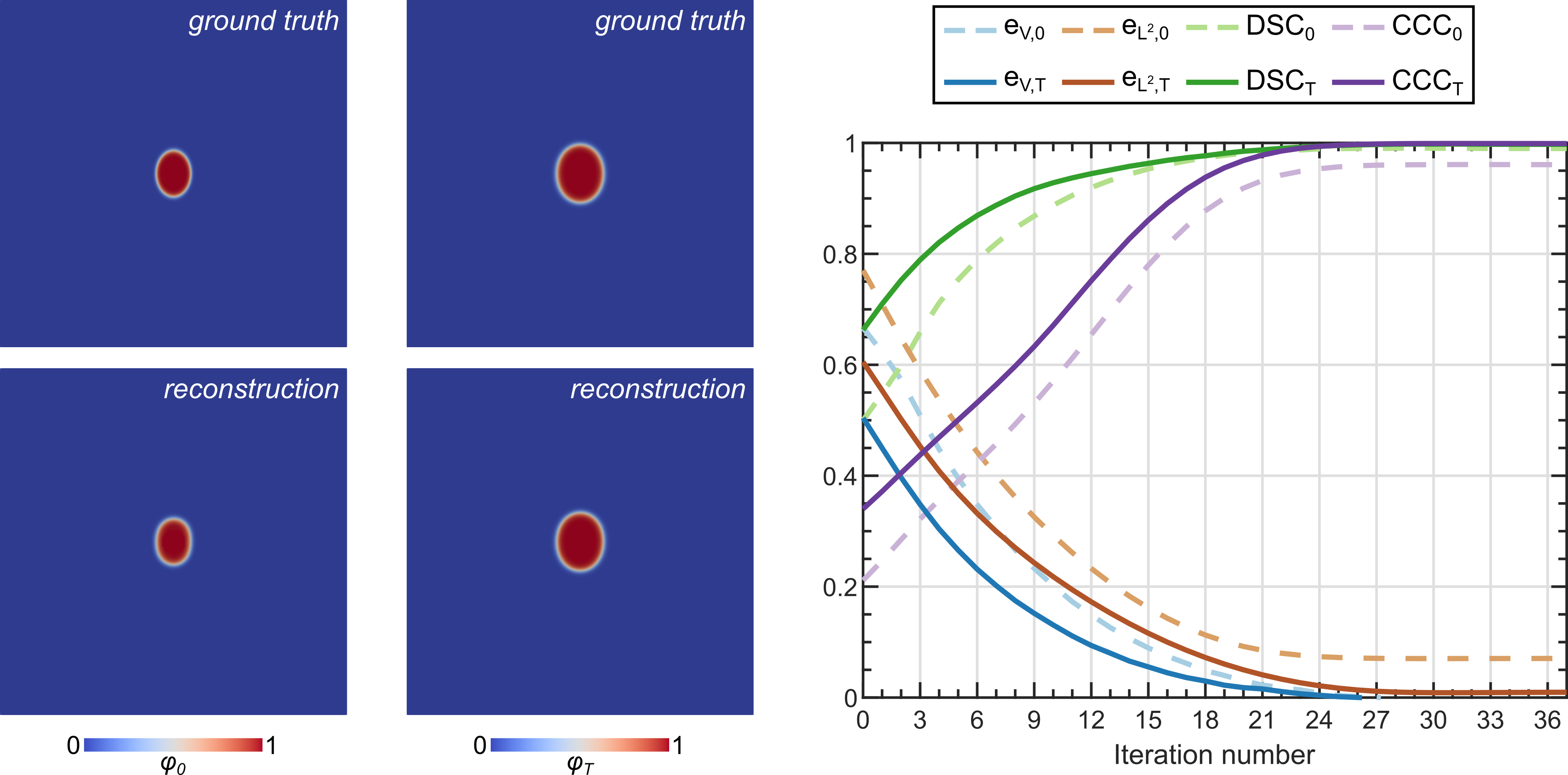} 
    \caption{Reconstruction of the initial tumour phase field from a measurement at $T=15$ days affected by $10\%$ Gaussian noise using the Landweber iteration scheme. The first two columns compare the tumour phase field from the reference simulation without noise and the corresponding reconstruction at $t=0$ and $t=T$. The plot in the last column provides the values of the four metrics used to assess the reconstruction of the tumour phase field in each iteration of the Landweber algorithm. }
    \label{fig:noise_landwebert15_metrics}
  \end{figure}

  \begin{figure}[!t]   
    \includegraphics[width=\textwidth]{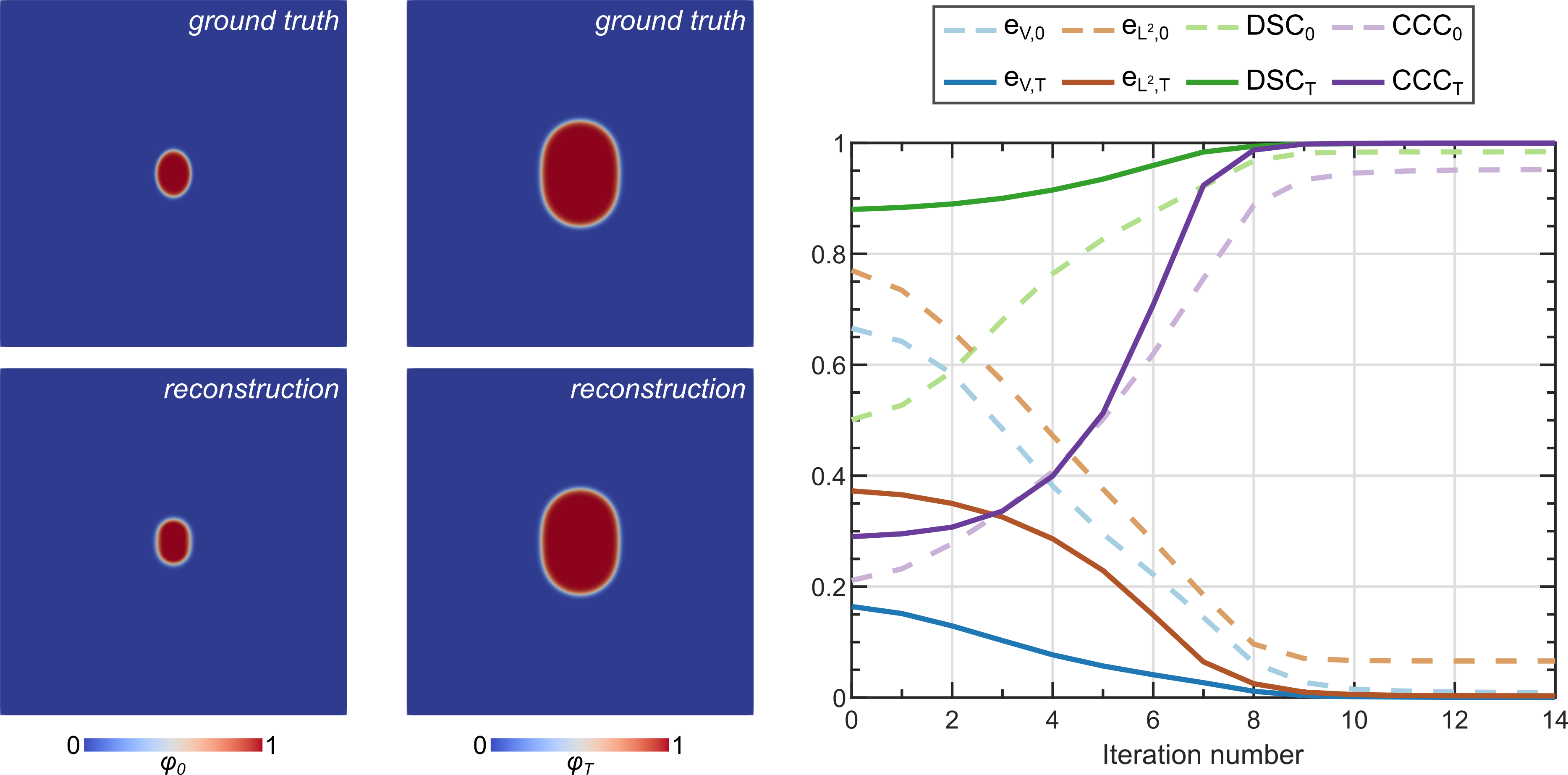} 
    \caption{Reconstruction of the initial tumour phase field from a measurement at $T=90$ days affected by $10\%$ Gaussian noise by leveraging the adaptive gradient descent algorithm. The first two columns compare the tumour phase field from the reference simulation without noise and the corresponding reconstruction at $t=0$ and $t=T$. The plot in the last column provides the values of the four metrics used to assess the reconstruction of the tumour phase field in each iteration of the adaptive gradient descent algorithm.}
    \label{fig:noise_agd90_metrics}
  \end{figure}

  \begin{figure}[!t]   
    \includegraphics[width=\textwidth]{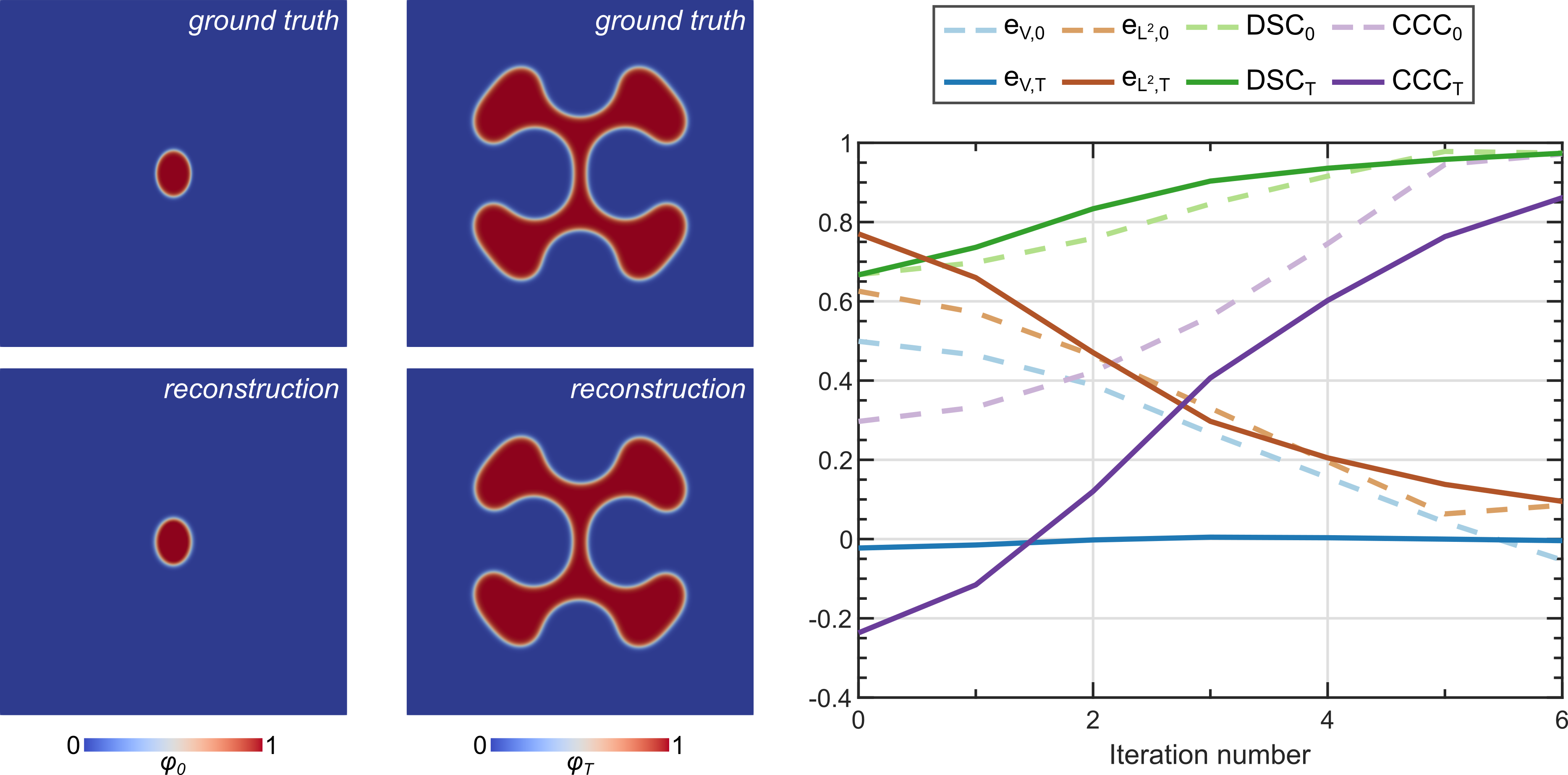} 
    \caption{Reconstruction of the initial tumour phase field from a measurement at $T=365$ days affected by $10\%$ Gaussian noise by leveraging the adaptive gradient descent algorithm. The first two columns compare the tumour phase field from the reference simulation without noise and the corresponding reconstruction at $t=0$ and $t=T$. The plot in the last column provides the values of the four metrics used to assess the reconstruction of the tumour phase field in each iteration of the adaptive gradient descent algorithm.}
    \label{fig:noise_agd365_metrics}
  \end{figure}

\begin{figure}[!t]
    \begin{subfigure}[t]{0.2\textwidth}
      \includegraphics[width=\textwidth]{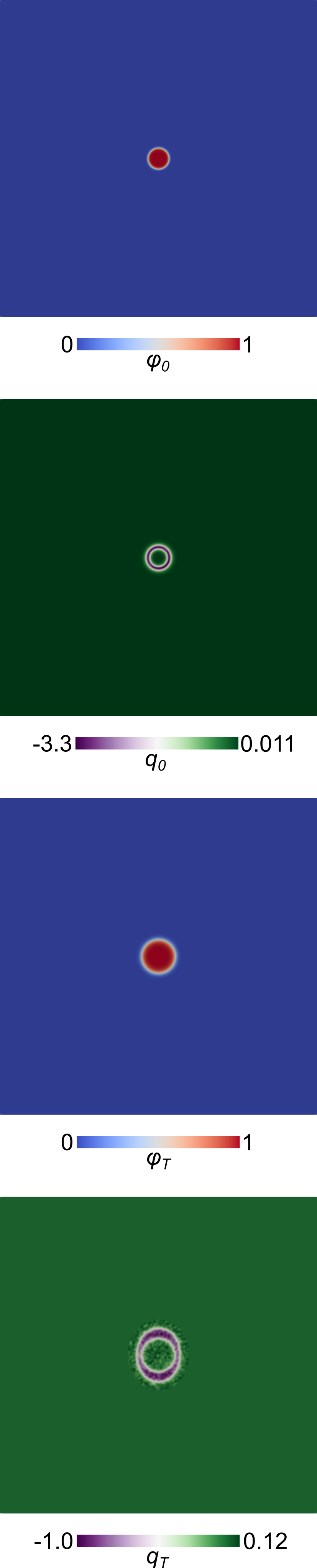}
      \caption{$j=0$}
    \end{subfigure}
    \hfill
    \begin{subfigure}[t]{0.2\textwidth}
      \includegraphics[width=\textwidth]{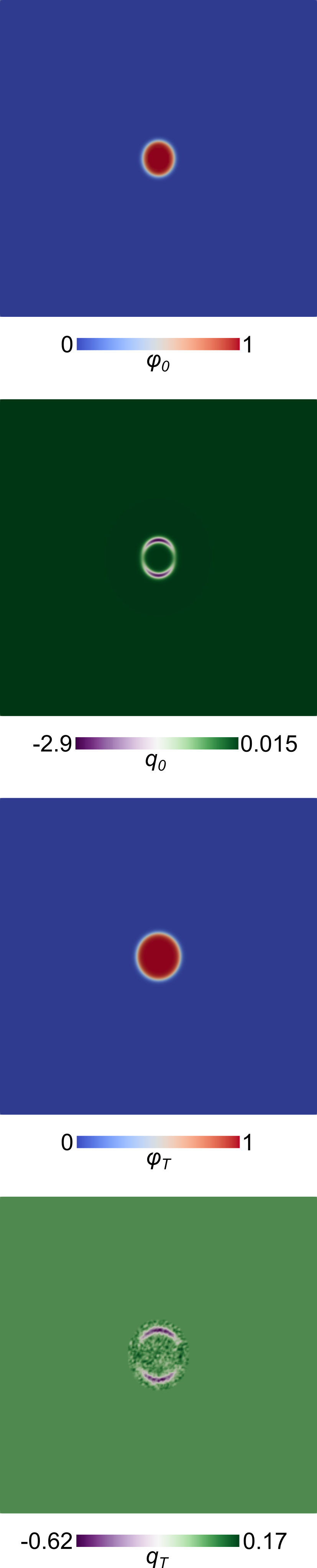}
      \caption{$j=11$}
    \end{subfigure}
    \hfill
    \begin{subfigure}[t]{0.2\textwidth}
      \includegraphics[width=\textwidth]{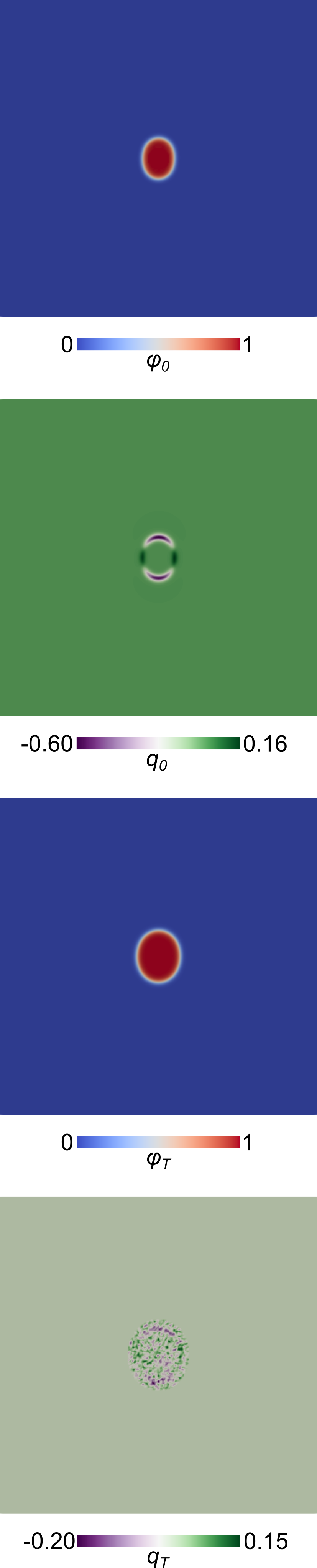}
      \caption{$j=22$}
    \end{subfigure}
    \hfill
    \begin{subfigure}[t]{0.2\textwidth}
      \includegraphics[width=\textwidth]{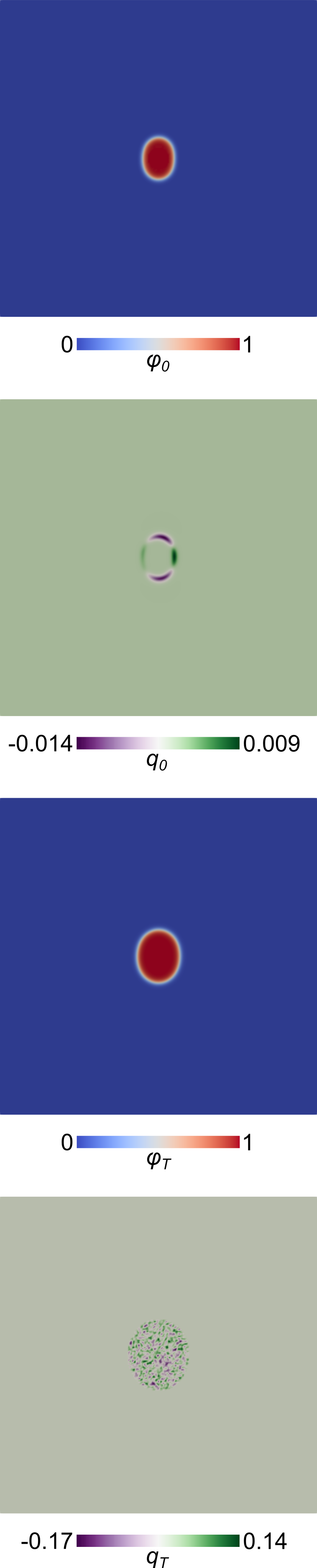}
      \caption{$j=37$}
    \end{subfigure}
    \hfill
    \caption{Reconstruction of the initial tumour phase field from a measurement at $T=15$ days affected by $10\%$ Gaussian noise by leveraging the Landweber iteration scheme. In each panel, the first two rows represent the tumour phase field and its corresponding adjoint variable at $t=0$ (i.e, $\phi_0$ and $q_0$), while the last two rows provide the same quantities at $t=T$ (i.e, $\phi_T$ and $q_T$).}
    \label{fig:noise_landwebert15_iters}
  \end{figure}

   \begin{figure}[!t]
    \begin{subfigure}[t]{0.2\textwidth}
      \includegraphics[width=\textwidth]{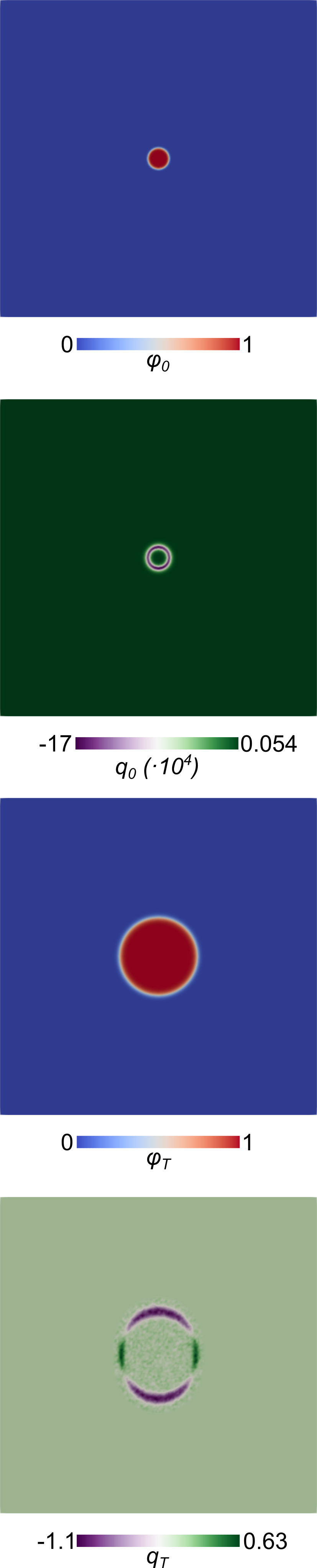}
      \caption{$j=0$}
    \end{subfigure}
    \hfill
    \begin{subfigure}[t]{0.2\textwidth}
      \includegraphics[width=\textwidth]{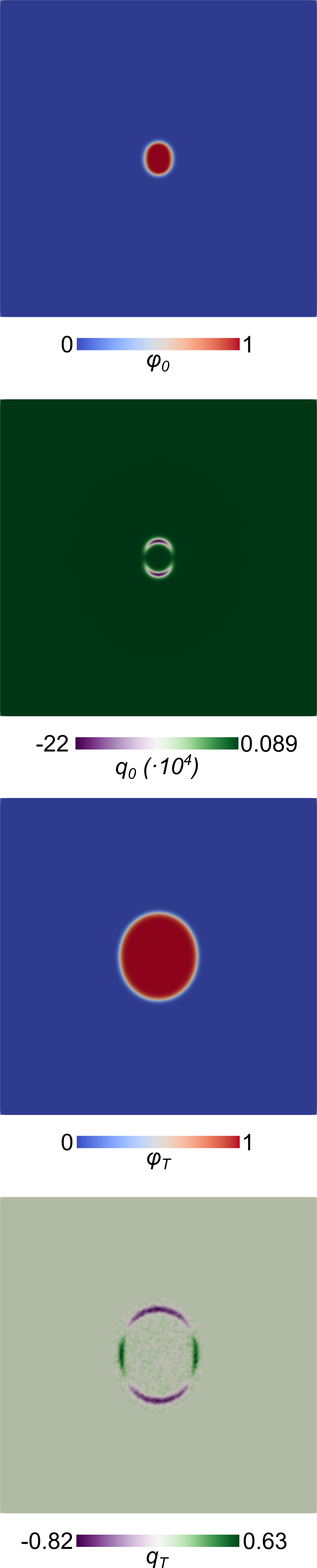}
      \caption{$j=5$}
    \end{subfigure}
    \hfill
    \begin{subfigure}[t]{0.2\textwidth}
      \includegraphics[width=\textwidth]{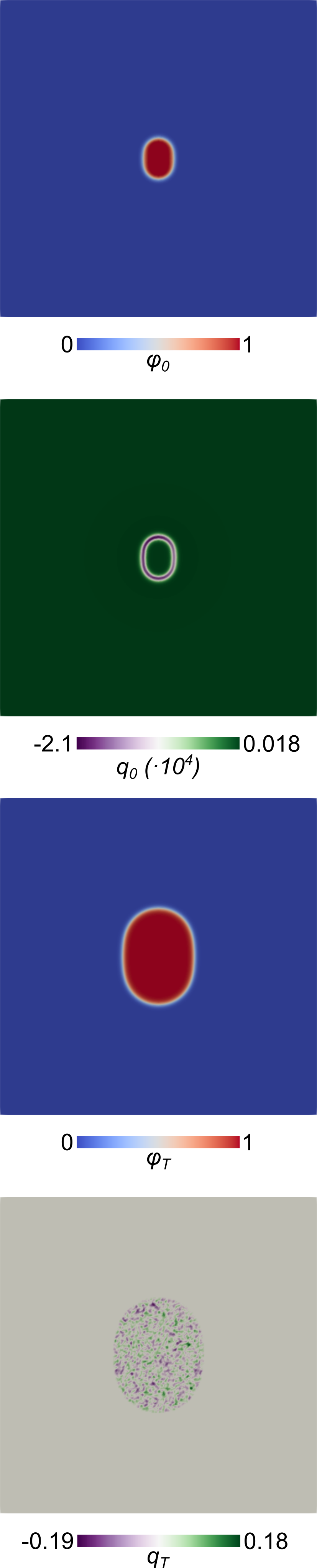}
      \caption{$j=9$}
    \end{subfigure}
    \hfill
    \begin{subfigure}[t]{0.2\textwidth}
      \includegraphics[width=\textwidth]{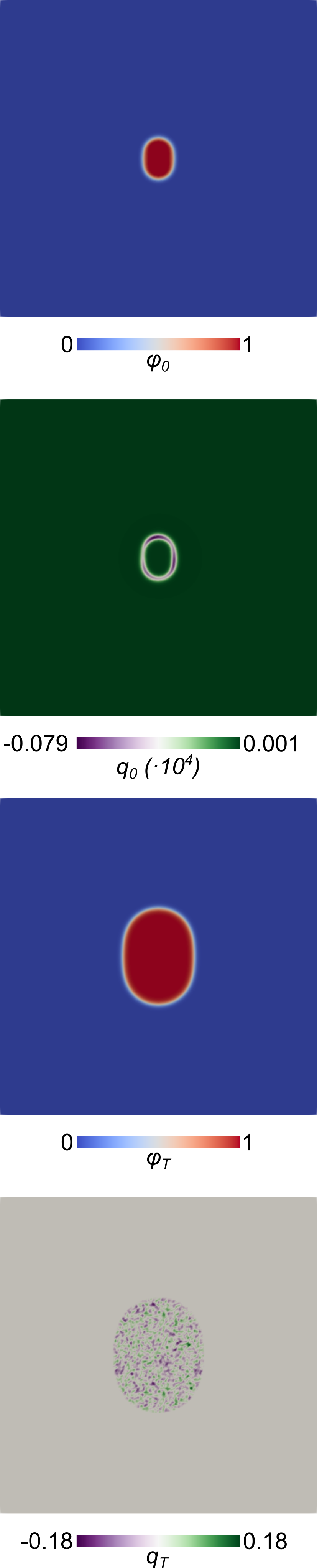}
      \caption{$j=14$}
    \end{subfigure}
    \hfill
    \caption{Reconstruction of the initial tumour phase field from a measurement at $T=90$ days affected by $10\%$ Gaussian noise by employing the adaptive gradient descent algorithm. In each panel, the first two rows represent the tumour phase field and its corresponding adjoint variable at $t=0$ (i.e, $\phi_0$ and $q_0$), while the last two rows provide the same quantities at $t=T$ (i.e, $\phi_T$ and $q_T$).}
    \label{fig:noise_agd90_iters}
  \end{figure}

 \begin{figure}[!t]
    \begin{subfigure}[t]{0.2\textwidth}
      \includegraphics[width=\textwidth]{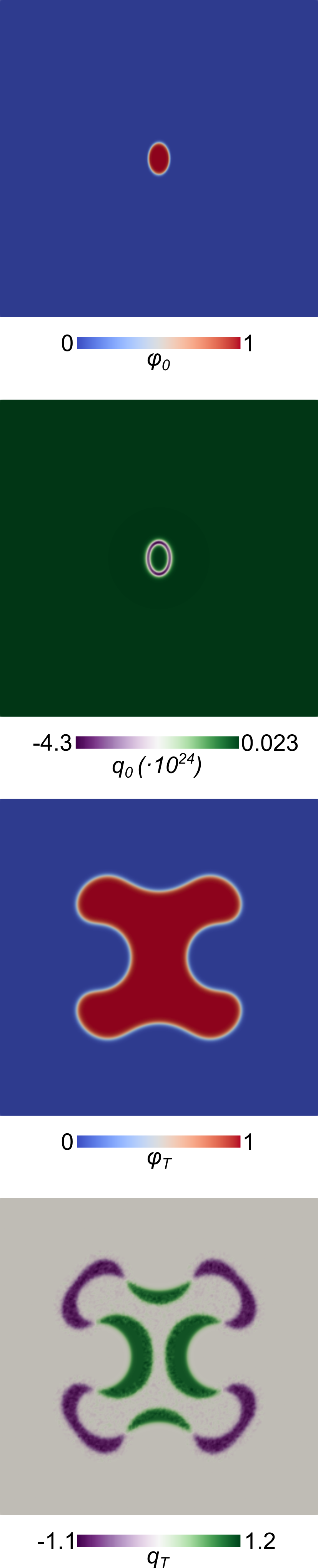}
      \caption{$j=0$}
    \end{subfigure}
    \hfill
    \begin{subfigure}[t]{0.2\textwidth}
      \includegraphics[width=\textwidth]{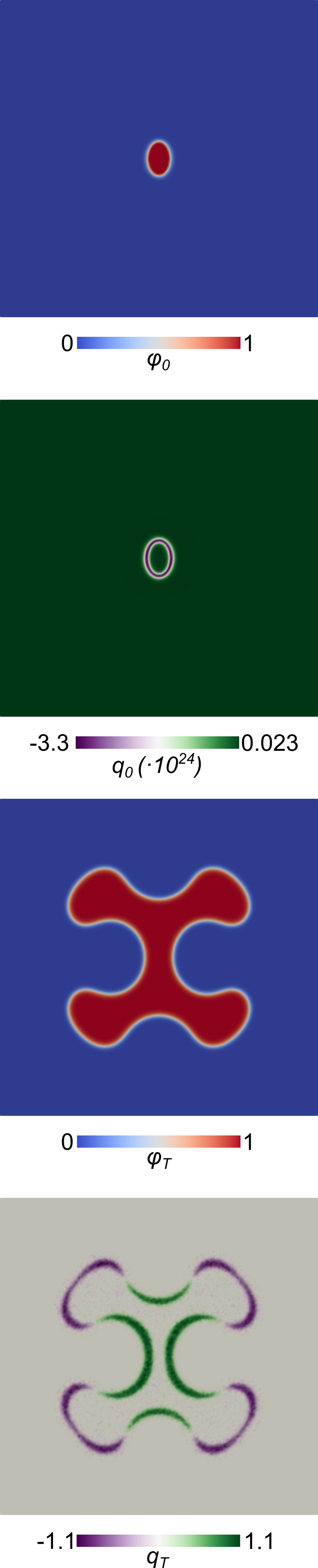}
      \caption{$j=2$}
    \end{subfigure}
    \hfill
    \begin{subfigure}[t]{0.2\textwidth}
      \includegraphics[width=\textwidth]{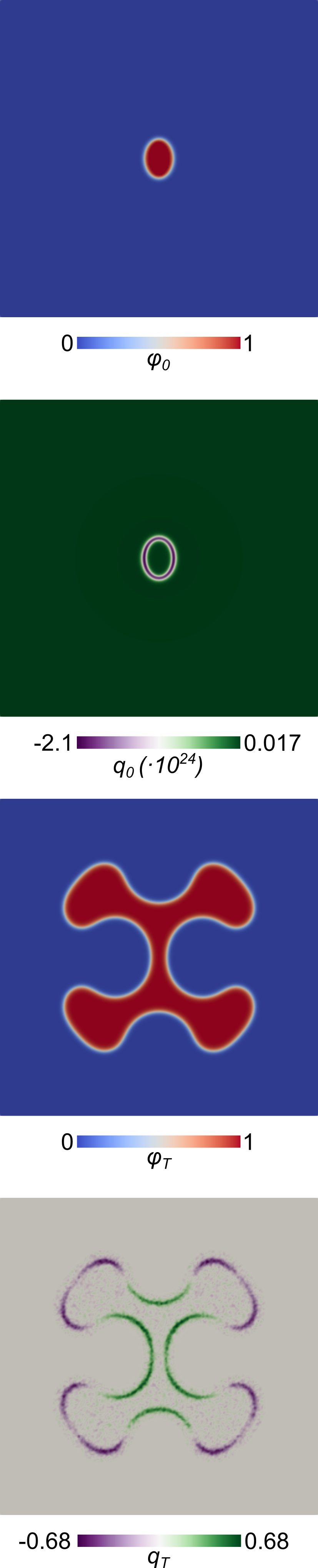}
      \caption{$j=4$}
    \end{subfigure}
    \hfill
    \begin{subfigure}[t]{0.2\textwidth}
      \includegraphics[width=\textwidth]{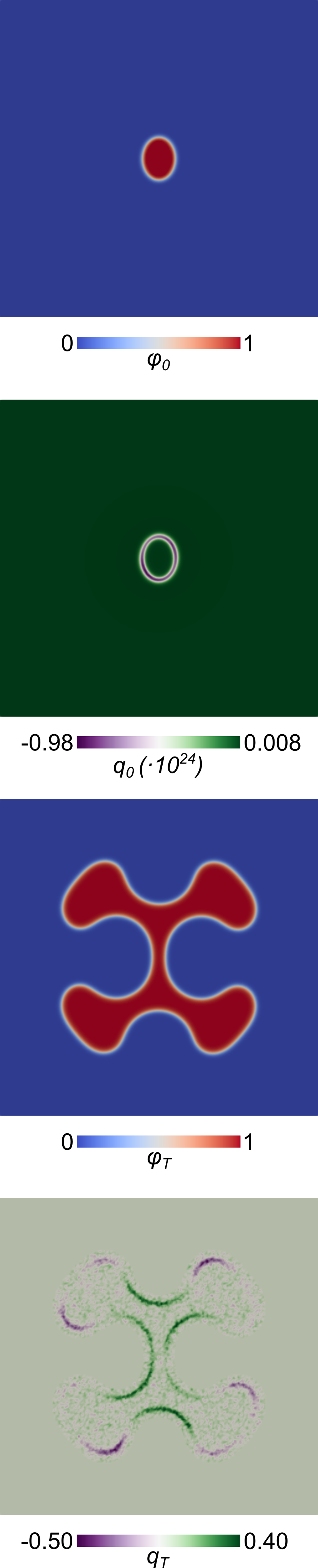}
      \caption{$j=6$}
    \end{subfigure}
    \hfill
    \caption{Reconstruction of the initial tumour phase field from a measurement at $T=365$ days affected by $10\%$ Gaussian noise by employing the adaptive gradient descent algorithm. In each panel, the first two rows represent the tumour phase field and its corresponding adjoint variable at $t=0$ (i.e, $\phi_0$ and $q_0$), while the last two rows provide the same quantities at $t=T$ (i.e, $\phi_T$ and $q_T$).}
    \label{fig:noise_agd365_iters}
  \end{figure}

In this section, we analyse the performance of the Landweber iteration scheme and the adaptive gradient descent algorithm in the presence of noise. 
Figure~\ref{fig:noise_landwebert15_metrics} shows the results for a time horizon $T=15$ days using the Landweber iteration scheme, including the tumour reconstruction at $t=0$ and $t=T$ along with the evolution of the metrics to assess the quality of the tumour reconstruction.
Additionally, Figures~\ref{fig:noise_agd90_metrics} and \ref{fig:noise_agd365_metrics} provide the corresponding results for time horizon $T=90$ and 365 days using the adaptive gradient descent algorithm. 
Similarly to the long-time horizon results presented in Section~\ref{res:longt}, for the latter scenario we adjusted the initial guess to an ellipsoid ($a=$\SI{100}{\micro\meter}, $b=$\SI{150}{\micro\meter}) and we further increased the tolerance $\varepsilon_{SD}$ to $0.1$ in order to achieve convergence.
In Figures~\ref{fig:noise_landwebert15_metrics}-\ref{fig:noise_agd365_metrics}, the values of the metrics at $t=T$ and $t=0$ are calculated with respect to the ground truth without noise to enable comparison with respect to the results in Sections~\ref{res:shortt} and \ref{res:longt}.
Moreover, Figures~\ref{fig:noise_landwebert15_iters}, \ref{fig:noise_agd90_iters}, and \ref{fig:noise_agd365_iters} show the changes in the tumour phase field and the dual variable $q$ at $t=0$ and $t=T$ during the initial tumour reconstruction with the Landweber iteration scheme ($T=15$ days) and the adaptive gradient descent algorithm ($T=90$ and 365 days).
The results in Figures~\ref{fig:noise_landwebert15_metrics}-\ref{fig:noise_agd365_metrics} are analogous to those presented in Figures~\ref{fig:landwebert15_metrics}, \ref{fig:agd90_metrics}, and \ref{fig:agd365_metrics} respectively.
Indeed, the values of the reconstruction metrics at $t=0$ and $t=T$ are virtually the same as those obtained in the corresponding scenarios without noise.
The presence of noise only impacts in a small increase in the number of iterations for convergence in the scenarios with $T=15$ and 90 days.
Additionally, the spatial maps obtained for $\phi_0$, $q_0$, $\phi_T$, and $q_T$ during the tumour reconstruction in the scenarios with a noisy $\phimeas$ align with those calculated in the corresponding scenarios without noise (see  Figures~\ref{fig:landwebert15_iters}, \ref{fig:agd90_iters}, and \ref{fig:agd365_iters}, respectively).
The main difference is observed for $q_T$, since the noisy $\phimeas$ affects its calculation in each iteration of the reconstruction algorithms.
We further observe that the presence of noise only impacts $q_0$ when the difference between $\phimeas$ and $\phi_T$ is in the neighbourhood of the value of the added noise. 
In this situation, $q_T$ essentially consists of noisy values that ultimately render an asymmetric or irregular $q_0$ map.
However, given that the update of the initial tumour phase field towards the end of the algorithm is minimal, these irregularities in $q_0$ do not impact the tumour reconstruction.
Thus, the qualitative and quantitative accuracy of the tumour reconstruction at time $t=0$ and $t=T$ is virtually the same despite the noise added to $\phimeas$.
These results demonstrate the robustness of the Landweber iteration scheme and the adaptive gradient descent method in recovering the initial ground truth under noisy measurements.

\section{Discussion}\label{sec:discussion}

In this work, we studied the inverse problem of reconstructing an earlier state of prostate cancer from a given spatial measurement collected at a certain time of interest $T$.
Within the inverse problem formulation, this instant becomes the time horizon and we used our previously presented phase-field model to recover prostate cancer growth before the measurement \cite{CGLMRR2019,CGLMRR2021,BCFLR2024}.
From a mathematical perspective, the inverse problem consists of reconstructing the initial data $(\phi_0, \sigma_0, p_0)$ starting from a single measurement $(\phimeas, \sigmameas, \pmeas)$ at the end-time $T$. 
We stress that such a problem is generally known to be severly ill-posed, in particular the larger the final time $T$ grows.
Leveraging the quantitative Lipschitz stability estimate on finite-dimensional subspaces found in \cite{BCFLR2024}, we motivated the use of a Landweber iteration scheme to approximate the solution of the inverse problem. 
We proved some theoretical convergence results of such a method in both its variants with a fixed  and an adaptive steepest descent step size. 
With the idea of accelerating its convergence, especially for the most challenging case of large time horizons $T$, we also employed a different choice of the step size, proposed in \cite{MM2019adaptive}. 
Even if lacking theoretical guarantees, due to the heavy use of convexity assumptions in \cite{MM2019adaptive}, this second method proved very powerful in tackling such challenging cases. 

To better understand the performance of the reconstruction algorithms, we conducted a simulation study considering a tumour growing in a square tissue patch and using synthetic ground truth data generated by the phase-field model.
Our results show that the Landweber iteration scheme yields a high-quality reconstruction of the initial tumour conditions for short time horizons (e.g., a few weeks).
These results also confirmed the infralinear convergence of the Landweber method that was derived analytically in Theorem~\ref{thm:landweber}.
Furthermore, the computational study presented in this work  demonstrated the capability of the adaptive gradient descent method to enable the reconstruction of tumours for longer time horizons (e.g., several months to one year).
Nevertheless, to achieve convergence in the longest time horizon considered in the study ($T=365$ days) we needed to adapt the initial guess to a closer geometry to the ground truth and reduce the convergence tolerance used in the other time horizon scenarios.
Of note, this last choice was supported by the results for shorter $T$ cases since they show that convergence is achieved before the reconstruction algorithm reaches the preset tolerance.
Additionally, the results obtained for the $T=365$  demonstrate an excellent reconstruction of the tumour, which are comparable to the results obtained with the adaptive gradient descent algorithm for $T=90$ and the preset tolerance.
To complete our simulation study, we further analysed the performance of both reconstruction methods under a noisy measurement at the time horizon.
Our simulation results in this situation show that both reconstruction algorithms can recover the initial conditions of the tumour with sufficiently high quality.
Thus, our simulation study suggests that (i) the Landweber iteration scheme might be preferred for the reconstruction of recent stages of a tumour (i.e., short $T$) due to the existence of robust convergence analytical guarantees, and (ii) the adaptive gradient descent algorithm can be alternatively leveraged to efficiently handle longer temporal reconstructions, although it may require using initial guesses closer to the ground truth and higher tolerances for convergence.

We believe that we have obtained encouraging results, especially considering that we treated a complex system in full generality, differently from the partial results obtained in the previously mentioned works \cite{JBJA2019, SSMB2020, subramanian2022ensemble}.
Nevertheless, this work also presents some limitations that can be addressed in future studies.
First, we only considered a particular phase-field model for tumour growth \cite{CGLMRR2019,CGLMRR2021,BCFLR2024}.
A more complete analysis of the reconstruction algorithms studied herein should also consider different kinds of phase-field models or models based on Fisher-Kolmogorov and biomechanical formulations, which are also commonly used in computational oncology \cite{Lorenzo2022_review, Wu2022, Hormuth2021, Wong2016, Vavourakis2018, Stylianopoulos2013}.
Additionally, we only considered untreated growth, which can be of interest for newly-diagnosed tumours \cite{Lorenzo2024}, but future studies could also address tumour reconstruction including treatment effects \cite{Lorenzo2022_review, Wu2022, Hormuth2021, CGLMRR2019}, which can be of interest when imaging data are acquired during therapy.
Second, the model parameters governing tumour growth were assumed to be known.
This might be feasible in preclinical cases (e.g., \emph{in vitro} and \emph{in vivo} animal studies) leveraging cancer cell lines with previously characterised dynamic features (e.g., proliferation and invasion rates) \cite{Yang2022,Lima2022,Burbanks2023}.
However, the application of tumour reconstruction methods in clinical scenarios will require the recovery of not only the initial conditions of the tumour, but also the main parameters governing its growth dynamics \cite{SSMB2020, Lorenzo2022_review}.
Of note, for such a challenging mathematical problem, it might be necessary to first perform a sensitivity analysis and a parameter identifiability study to find a small set of significant parameters that can be effectively reconstructed \cite{Lorenzo2023, Craig2023}.
Third, our simulation study focused on a 2D square tissue patch.
Although this computational setup facilitated the analysis of the reconstruction algorithms, future studies should also investigate their performance in 3D scenarios and consider patient-specific  organ anatomies extracted from the reference imaging measurement at the time horizon \cite{Lorenzo2016,Lorenzo2017, Lorenzo2024, Hormuth2021, Wong2016, Wise2008, Xu2016}.
Finally, we used synthetic data generated via simulation of our phase-field model to generate the ground truth.
While this approach enables preliminary confirmation that the proposed reconstruction methods can recover spatiotemporal tumour maps matching the model dynamics, their ultimate validation needs to employ real-world data (e.g., magnetic resonance imaging, computerised tomography) \cite{Lorenzo2022_review}. 

Despite the formidable mathematical and computational challenges involved in the reconstruction of early tumour stages from a unique spatial measurement, the advances in the understanding of this problem could ultimately have a profound impact on the use of computational tumour forecasts in clinical scenarios.
In particular, the accurate personalised prediction of tumour growth from a single imaging dataset at diagnosis is a long-standing challenge in the field of computational oncology, which could dramatically improve treatment planning for better therapeutic outcomes \cite{Lorenzo2022_review, CGLMRR2021, Jarrett2018, Lipkova2019} as well as patient triaging to adequate management options \cite{Lorenzo2024, Brady2020, Yankeelov2024}.
Thus, efficient and robust tumour reconstruction methods can be a key computational asset in the design of digital twins to optimise cancer monitoring and treatment \cite{Wu2022a,Hernandez2021,Yankeelov2024}, thereby contributing towards a more predictive and personalised paradigm in clinical oncology.

	\bigskip
	
	\noindent\textbf{Acknowledgements.}
C. Cavaterra, M. Fornoni and E. Rocca have been partially supported by the MIUR-PRIN Grant 2020F3NCPX 
	``Mathematics for industry 4.0 (Math4I4)''. 
C. Cavaterra has been partially supported by the MIUR-PRIN Grant 2022  
	``Partial differential equations and related geometric-functional inequalities''. 
C. Cavaterra, M. Fornoni and E. Rocca are members of  
	GNAMPA (Gruppo Nazionale per l'Analisi Matematica, la Probabilit\`a e le loro Applicazioni)
	of INdAM (Istituto Nazionale di Alta Matematica). 
 The research of C. Cavaterra is part of the activities of ``Dipartimento di Eccellenza 2023-2027'' of Universit\`a degli Studi di Milano. Elena Beretta's research has been partially supported by NYUAD Science Program Project Fund AD364. 
 E. Rocca also acknowledges the support of Next Generation EU Project No.P2022Z7ZAJ (A unitary mathematical framework for modelling muscular dystrophies).
 G. Lorenzo acknowledges the support of a fellowship from ‘‘la Caixa” Foundation (ID 100010434). The fellowship code is LCF/BQ/PI23/11970033. We also thank the Texas Advanced Computing Center (TACC) for providing high-performance computational resources that contributed to the results presented in this work.
 
	\bigskip
 
\footnotesize

\end{document}